\documentclass[11pt,reqno]{amsart}
\usepackage[T1]{fontenc}
\usepackage{amsmath}
\usepackage{graphicx}
\usepackage{cite}
\usepackage{fixltx2e}
\usepackage{color}
\usepackage[normalem]{ulem}

\usepackage{color}
\definecolor{MyLinkColor}{rgb}{0,0,0.4}
\numberwithin{equation}{section}

\newcommand{\im}{\mathop{\rm Im}\nolimits}

\newcommand{\tr}{\mathop{\rm tr}\nolimits}

\newcommand{\HH}{\mathcal{H}}
\newcommand{\h}{\rho}
\newcommand{\0}{\Omega}
\newcommand{\e}{\varepsilon}

\newcommand{\p}{\partial}
\newcommand{\wt}{\widetilde}
\newcommand{\ov}{\overline}

\newcommand{\cO}{\mathcal{O}}
\newcommand{\D}{{\mathcal D}}

\newcommand{\kF}{\mathcal{F}}

\newcommand{\cK}{{\mathcal K}}
\newcommand{\kL}{\mathcal{L}}

\newcommand{\R}{\mathbb{R}}

\newcommand{\N}{\mathbb{N}}

\newcommand{\X}{\mathbb{X}}
\newcommand{\Y}{\mathbb{Y}}

\newtheorem{thm}{Theorem}[section]
\newtheorem{prop}[thm]{Proposition}
\newtheorem{lemma}[thm]{Lemma}

\newtheorem{defn}[thm]{Definition}

\theoremstyle{remark} 
\newtheorem{rem}[thm]{Remark}
\newtheorem{ex}[thm]{Example}

\numberwithin{equation}{section}

\title[Stratified water waves with singular density gradients]{Stratified periodic water waves  with singular \\ density gradients}

\author[J. Escher]{Joachim Escher}
\address{Institut f\"ur Angewandte Mathematik, Leibniz Universit\"at Hannover, Welfengarten~1, 30167 Hannover, Deutschland.}
\email{escher@ifam.uni-hannover.de}

\author[P. Knopf]{Patrik Knopf}
\address{Fakult\"at f\"ur Mathematik, Universit\"at Regensburg,   93053 Regensburg, Deutschland.}
\email{patrik.knopf@ur.de}
\email{bogdan.matioc@ur.de}

\author[C. Lienstromberg]{Christina Lienstromberg}
\address{Institut f\"ur Angewandte Mathematik, Universit\"at Bonn, Endenicher Allee~60, 53115 Bonn, Deutschland.}
\email{lienstromberg@iam.uni-bonn.de}

\author[B.--V. Matioc]{Bogdan--Vasile Matioc}


\usepackage[colorlinks=true,linkcolor=MyLinkColor,citecolor=MyLinkColor]{hyperref} 

\begin{document}

\begin{abstract}
We consider Euler's equations for free surface waves traveling on a body of density stratified water in the scenario when 
gravity and surface tension act as restoring forces. 
The flow is continuously stratified, and the water  layer is bounded from below by an impermeable horizontal bed. 
For this problem we establish  three equivalent  classical formulations in a suitable setting of strong solutions 
which may describe nevertheless waves with singular density gradients. 
Based upon this equivalence we then   construct two-dimensional symmetric periodic traveling waves that are monotone between each crest and trough.
Our analysis uses, to a large extent, the availability of a weak formulation of the water wave problem, the regularity properties of the corresponding weak solutions, 
and methods from nonlinear functional analysis.   
\end{abstract}

\maketitle

 
\section{Introduction}\label{Sec:0}
Stratification is a phenomenon that is common in ocean flows where in the presence of salinity and under the influence of 
the gravitational force a heterogeneity in the fluid is produced.
Stratification corresponds to the formation of fluid layers, normally arranged horizontally with the less dense layers  being located on top of the denser ones.
 This phenomenon  may be caused  by many other factors including temperature, pressure, topography and oxygenation.
Because of the plethora of effects resulting from stratification, such flows have received  much attention especially in geophysical fluid dynamics.
In the setting of traveling stratified waves  the problem is modeled by the stationary Euler equations 
 for incompressible fluids, subject to   natural boundary conditions, cf. \eqref{eq:1}.
The study of two-dimensional stratified flows  dates back to the pioneering work of Dubreil-Jacotin.
In 1937 Dubreil-Jacotin \cite{DJ37} constructed small-amplitude stratified traveling gravity waves by using power series expansions. 
Previously in \cite{DJ32} she showed that Gerstner's explicit solution \cite{Co01, Ge09} can be accommodated to describe exact traveling gravity 
waves with an arbitrary stratification. Furthermore, related to Gerstner's solution, there is a further exact solution describing 
an edge wave propagating along a sloping beach \cite{Ra11, C01a, MA13} allowing even for an arbitrary stratification.
Recently also other exact and explicit solutions of stratified flows in different geophysical regimes have been found, cf. \cite{HM19, HM19a, HM18, CJ16, Con13, CJ12, Kl19}. 

Many of the papers dedicated to the  stratified water   wave problem consider the vertical stratification to be fairly smooth. 
Small amplitude periodic gravity water waves possessing a linear stratification  have been constructed in \cite{EMM11}. 
These waves may contain critical layers and stagnation points and the authors in \cite{EMM11} provide
also the qualitative picture of the flow beneath the constructed waves.
Small amplitude periodic capillary-gravity waves with  sufficiently regular density which may still contain critical layers have been found in \cite{HM13} by means of local bifurcation. 
The local bifurcation branches of solutions to the stratified water wave problem have been extended by using
global bifurcation theory to global branches in \cite{HM14b}.
The papers \cite{EMM11, HM13, HM14b} use the  Long-Yih formulation \cite{L53, Y60} (see \eqref{eq:P}) of the problem whose availability is facilitated by the fact that the density is sufficiently regular.
When excluding critical layers and stagnation points the stratified wave problem can be considered by using Dubreil-Jacotin's formulation (see \eqref{eq:Q}-\eqref{C4}).
This approach has been followed in \cite{Wal09, Wa14a, Wal14b} where -- by means of local and global bifurcation theory -- small and large amplitude
 stratified periodic water waves of finite depth are constructed both in the presence and absence of surface tension.
The existence of solitary free surface water waves with general regular density distribution, together with a qualitative study of such flows,
 has been provided only recently \cite{WAl18}, using again Dubreil-Jacotin's formulation in their treatise.
Qualitative properties of stratified water waves with regular density, such as symmetry, regularity, and the unique determination of the wave when knowing the pressure on the bed and the fluid stratification, have been  addressed in \cite{WAl18b, HM, W13, Wal09a}.

In  ocean flows, however, the  density varies  strongly in thin layers called pycnoclines which exhibit  sharp density gradients, cf., e.g., \cite{Pe79, CB94, P13}.
For this reason some of the research \cite{CB94, Ma12a, Mar16, Mar17, CI15, CIM16, T80,  ASW16, AT89} is restricted to so-called  layered models which consider  the flow as consisting of a finite
number of vertical layers each of them having uniform density. 
These layers are separated by internal waves which are mainly driven by the density difference between the layers (some models  also consider surface tension effects). 
  In this  paper we consider a continuous stratification, but  
 allow for solutions with a density gradient which
  is merely $L_r$-integrable
  with $r\in(1,\infty)$ arbitrarily close to $1$.
  A similar setting  has been studied in \cite{WC16} but in the absence of surface tension forces. 
  The authors of \cite{WC16} deal with a layered model with the density in each layer  varying continuously in such a way that the density gradient 
  is $L_r$-integrable, but with $r>2$.
   The choice $r>2$ is related to the Sobolev embedding $W^1_r(\R^2)\hookrightarrow {\rm C}^{1-2/r}(\R^2).$
In this regime the different formulations  of the water wave problem mentioned above are equivalent in the setting of periodic Sobolev solutions.  
Our first  main result is an equivalence result for the three formulations of the problem in a suitable
 setting of strong solutions, cf. Theorem \ref{T:Ep} (see Theorem \ref{T:E0} for the case when surface tension is neglected).  
 The equivalence in these theorems holds for  $r\in[1,\infty)$. 
For $r\in[1,2]$ the Sobolev regularity is too weak for the equations to be realized in $L_r$-spaces, and therefore our notion of strong solutions involves some complementary 
H\"older regularity. 
Here the H\"older exponent $\alpha=1-1/r\in[0,\infty)$ results from the embedding $W^1_r(\R)\hookrightarrow {\rm C}^{1-1/r}(\R).$
  The main result Theorem~\ref{MT1}, which relies on the equivalence in Theorem~\ref{T:Ep}, establishes the existence of infinitely many periodic solutions to the stratified water wave problem
  having merely a $L_r$-integrable density gradient.
  Moreover, the wave profiles  are symmetric with respect to crest and through lines and strictly monotone in between them. 
 The proof of Theorem \ref{MT1} uses the Crandall--Rabinowitz theorem \cite[Theorem~1.7]{CR71} on  bifurcation  from simple eigenvalues
  in the context of a weak interpretation of Dubreil-Jacotin's formulation. 
For traveling water waves this idea was first used by Constantin and Strauss in \cite{CS11} to construct homogeneous periodic gravity water waves with discontinuous vorticity.
The situation of heterogeneous water waves is slightly different as the equations in the bulk cannot be recast in divergence form, see also \cite{WC16}.
Due to the presence of surface tension we need to deal with a second order nonlinear  equation on the surface boundary  corresponding to the dynamic boundary condition at the waves surface.
 Here we use a recent trick employed first in \cite{MM14, CM14b} (a similar idea appears also in \cite{M14, HM14b, ASW16}) to transform this equation into a Dirichlet boundary condition perturbed by a nonlinear and nonlocal part of order $-1$. 
A particular feature of our analysis is that we fix both the depth of the fluid as well as  the volume of  the fluid (within a period). 
This fact in combination with the weak regularity of the density gradient reduces the number of possible bifurcation parameters. 
For this reason the best (probably only) choice for a bifurcation parameter 
is the wavelength $\lambda$  (the wavelength has also been used in \cite{EW15} as one of the bifurcation  parameters). 
It is worth pointing out that this choice provides a remarkable identity in  (see \eqref{LEM:DW} below)
 that leads us to a very simple and elegant dispersion relation, cf. Lemma \ref{L:39}.

The paper is organized as follows: In Section \ref{Sec:1} we introduce the three formulations of the problem and we establish their equivalence in Theorems \ref{T:Ep} (and Theorem \ref{T:E0}).
Moreover, we state our main result in Theorem \ref{MT1} on the existence of laminar and nonlaminar flow solutions and some
qualitative properties. 
In Section \ref{Sec:2} we first introduce the notion of a weak solution to Dubreil-Jacotin's formulation and establish, 
by means of a shooting method, the existence of at least one laminar flow solution to this latter formulation. 
This solution does not depend on the horizontal variable, having thus parallel and flat streamlines, and it solves the problem for each value of $\lambda$.   
This set of laminar solutions (we have a solution for each $\lambda>0$) is then seen as the trivial branch of solutions to the problem. 
Merely the existence of the laminar solution imposes some restriction on the physical properties  of the flows, cf. \eqref{RES2} and Example~\ref{Ex:1}. 
 In Section \ref{Sec:3} we reformulate the equations as an abstract bifurcation problem and identify, by using methods  from  nonlinear functional analysis,
  a particular value $\lambda_*$ of the wavelength parameter where a local branch of nonlaminar weak solutions arises from the set of laminar solutions.
 For this we need to impose a further, quite explicit restriction in \eqref{RES3}.
  The proof of Theorem \ref{MT1} is then completed by showing that the weak solutions that were found are in fact strong solutions, cf. Proposition \ref{P:51}. 
  This gain of regularity relies on the regularity result in Theorem, \ref{T:RR} which is inspired by ideas presented in \cite{EC11} and \cite{EM14}.

 \section{Mathematical formulations and the main results}\label{Sec:1}
We now present three classical formulations of the steady water wave problem for stratified fluids.
We start with the classical Euler formulation. 
The motion of an  inviscid, incompressible, and stratified fluid is described by the Euler equations
\begin{equation*} 
\left\{
\begin{array}{rllllll}
\h (u_t+   u u_x+  vu_y)&=&-P_x \\
\h (v_t+  uv_x+ vv_y)&=&-P_y-g\h \\
\h_t+ \h_x u+\h_yv&=&0\\
u_x+v_y&=&0,
\end{array}
\right.
\end{equation*}
where $\h$ is the fluid's density, $u$ is the horizontal velocity, 
$v$ is the vertical velocity, $P$ is the pressure, and $g$ is the gravitational acceleration. 
The fluid domain is bounded from below by the impermeable flat bed $y=-d$,
where $d$ is a  fixed positive constant, and $y=\eta(t,x)$ denotes the wave surface.  
In addition to the conservation of momentum which is expressed by the first two equations of the system, the fluid is assumed to be incompressible and mass conserving (these properties correspond to the third and the fourth equation, respectively).
Our analysis is restricted to the physically relevant case of positive density, that is we assume throughout this paper that there exists a constant 
$\rho_0>0$  such that
\begin{align}\label{CR}
\rho\geq \rho_0.
\end{align}

The equations in the fluid domain are subject to the following boundary conditions 
\begin{equation*} 
\left\{
\begin{array}{rlllllll}
P&=&-\displaystyle\frac{\sigma\eta_{xx}}{(1+\eta_x^2)^{3/2}}&\text{on $y=\eta(t,x),$}\\
v&=&\eta_t +u\eta_x&\text{on $y=\eta(t,x),$}\\
v&=&0&\text{on}\ y=-d,
\end{array}
\right.
\end{equation*}
where the atmospheric pressure is set to  zero and $\sigma\geq0$ denotes the surface tension coefficient.
As we are interested in periodic waves we introduce
the positive constant $\lambda$ to denote the associated (minimal) wavelength.
Moreover, we require that
\begin{equation*} 
\int_0^\lambda \eta(t,x)\, dx=0 
\end{equation*}
at each time $t$. 
This condition means in particular that the fluid volume is fixed in each period of the fluid domain. 
Traveling periodic waves correspond to solutions of the previously introduced equations that exhibit a $(t,x)$-dependence 
of the form 
\begin{equation*} 
(u,v,P,\rho)(t,x,y)=(u,v,P,\rho)(x-ct,y)\qquad\text{and}\qquad \eta(t,x)=\eta(x-ct),
\end{equation*}
where $c>0$ is the wave speed, and which are $\lambda$-periodic in $x$.  
Observed from a frame that moves with the wave speed  $c$, traveling waves appear to be steady and we are left with the 
free boundary value problem
\begin{equation}\label{eq:1}
\left\{
\begin{array}{rllllll}
\h (u-c) u_x+\h vu_y&=&-P_x&\text{in $ \0_\eta,$}\\
\h(u-c)v_x+\h vv_y&=&-P_y-g\h&\text{in $ \0_\eta,$}\\
 (u-c)\h_x +v\h_y&=&0& \text{in $ \0_\eta,$}\\
u_x+v_y&=&0& \text{in $ \0_\eta,$}\\
P&=&-\displaystyle\frac{\sigma\eta''}{(1+\eta'^2)^{3/2}}&\text{on $y=\eta(x),$}\\
v&=& (u-c)\eta'&\text{on $y=\eta(x),$}\\
v&=&0&\text{on $y=-d,$}\\
\displaystyle\int_0^\lambda \eta(x)\, dx&=&0,
\end{array}
\right.
\end{equation}
where
 $$\0_\eta:=\{(x,y)\,:\,x\in\R,\, -d<y<\eta(x)\}.$$
 We point out that since $u$ and $c$ appear only in terms of the difference $u-c$ in \eqref{eq:1}, we may view the quintuplet
 $(u-c,v,P,\rho,\eta)$ as being the unknown, each of these functions additionally being $\lambda$-periodic with respect to the 
 horizontal variable $x$.

In order to study problem \eqref{eq:1} analytically it is useful to consider equivalent formulations.
To this end we define the so-called stream function $\psi$ by the relations
 \[
 \nabla\psi:=(-\sqrt\h v,\sqrt{\h}(u-c)) \, \, \, \text{in $\0_\eta$}\qquad\text{and}\qquad \psi=0 \, \, \, \text{on $y=\eta(x)$.}
 \]
One may observe that in the moving frame the streamlines of the flow coincide
with the level curves of the stream function.
Moreover, the density $\rho$ as well as the total hydraulic head
\[ 
E:=P+\h\frac{(u-c)^2+v^2}{2}+g\h y\qquad\text{in $\0_\eta$}
\]
are both constant along the streamlines.
In particular, if we require that 
\begin{align}\label{C1}
 \sup_{\0_\eta}(u-c)<0,
\end{align}
a condition which is a priori satisfied for homogeneous irrotational water waves, 
the hodograph transformation $\HH: \ov\0_\eta\to \ov\0$ defined by
\begin{equation}\label{HHH}
\HH(x,y):=(q(x,y),p(x,y)):=(x,-\psi(x,y))
\end{equation}
is a bijection. 
Here $\0:=\R\times (p_0,0) $ and  $p_0:=-\psi\big|_{y=-d}$ is a negative constant.
Using this property one can find two functions $\ov\h, \, \beta:[p_0,0]\to\R$, the so-called streamline density 
function and the Bernoulli function, respectively,
such that 
\[\text{$\rho\circ  \HH^{-1} =\ov\h$\quad and\quad $-\p_p (E\circ \HH^{-1})=\beta$.}\]
In particular $\h(x,y)=\ov\h(-\psi(x,y))$ in $\0_\eta.$
As the density usually increases with depth we restrict our considerations to the stably stratified regime defined by the
 inequality\footnote{The assumption that the density is non-decreasing with depth is not needed for the equivalence result in Theorem \ref{T:Ep} (or Theorem \ref{T:E0}), 
but is used to a large extend in the proof of the bifurcation result in Theorem \ref{MT1}.} 
\begin{align}\label{RES1}
\ov\rho'\leq 0.
\end{align} 

These considerations lead one to the  Long-Yih \cite{L53, Y60} formulation of the hydrodynamical problem~\eqref{eq:1}:
\begin{equation}\label{eq:P}
\left\{
\begin{array}{rlllllll}
 \Delta \psi &=&g y\ov \h'(-\psi)+\beta(-\psi)& \text{in $ \0_\eta$}, \\
 \psi&=&0 & \text{on $ y=\eta(x),$}\\
 \psi&=&-p_0 & \text{on $y=-d,$}\\
 \displaystyle |\nabla \psi|^2 -\frac{2\sigma\eta''}{(1+\eta'^2)^{3/2}}+2g\ov\h(0)y&=&Q &\text{on $y=\eta(x).$}
\end{array}
\right.
\end{equation}
 Under the assumption \eqref{CR} of positive density   the condition \eqref{C1} is equivalent  to
\begin{align}\label{C2}
  \sup_{\0_\eta}\psi_y<0.
\end{align}
The constant $Q$ in the equation $\eqref{eq:P}_4$ is related to the energy $E$. 
The equation $\eqref{eq:P}_4$ identifies $Q$ for waves with zero integral mean as follows
\begin{align}\label{C3}
 Q=\frac{1}{\lambda}\int_{0}^\lambda|\nabla \psi|^2(x,\eta(x))\, dx.
\end{align}

Using the (partial) hodograph transformation $\HH$, a further equivalent formulation of \eqref{eq:1}--\eqref{C1} 
may be derived in terms of the height function $h:\ov\0\to\R$ that is defined by  
\begin{align}\label{HF}
h(q,p)=y+d\qquad\text{for $(q,p)\in\ov\0.$}
\end{align}
The system \eqref{eq:P} can then be recast 
in the fixed rectangular domain $\Omega$ in the following form
\begin{equation}\label{eq:Q}
\left\{
\begin{array}{rrllll}
(1+h_q^2)h_{pp}-2h_qh_ph_{pq} +h_p^2h_{qq}-[g \ov\h' (h-d)+\beta]h_p^3&=&0 & \text{in $\0$},\\[1ex]
h&=&0 & \text{on } p=p_0,\\
 1+h_q^2 +h_p^2\Big[2g \ov\h(0)(h-d)- \displaystyle 
 \frac{2\sigma h_{qq}}{(1+h_q^2)^{3/2}}-\frac{1}{\lambda}\int_{0}^\lambda\frac{1+h_q^2}{h_p^2}(q,0)\, dq\Big]&=&0 & \text{on } p=0,
\end{array}
\right.
\end{equation}
 the relation \eqref{C1}  taking the form
 \begin{equation}\label{C4}
\inf_{\0}h_p>0.
\end{equation} 
In virtue of \eqref{C4} the quasilinear equation $\eqref{eq:Q}_1$ is uniformly elliptic.
 This equation is  complemented by a  nonlinear and nonlocal boundary condition on $p=0$ and a 
homogeneous Dirichlet condition on $p=p_0$. 
This formulation gives an insight into the flow as the streamlines in the moving frame are parametrized by the mappings 
$[x\mapsto h(x,p)-d].$
In the setting of classical solutions it is not difficult to show that the three formulations 
\eqref{eq:1}--\eqref{C1}, \eqref{eq:P}--\eqref{C3}, and \eqref{eq:Q}--\eqref{C4} are equivalent, cf., e.g.,
\cite{Con11, Wa14a, Wal14b, Wal09}.
This feature remains true in the more general framework described below.\medskip

\noindent{\bf Equivalent formulations.}  In Theorem \ref{T:Ep} we  present our first main result which establishes for capillary-gravity stratified water waves, 
that is for $\sigma>0$, the equivalence of the three formulations  in a suitable setting of  strong solutions. 
The case $\sigma = 0$ is treated in Theorem \ref{T:E0}.
A strong solution of any of the three formulations possesses weak derivatives up to highest order (the order is required by the equations) that are $L_r$-integrable.
Moreover the lower order derivatives enjoy some
additional H\"older regularity to ensure that all equations are satisfied  in $L_r$-spaces (in particular pointwise a.e.).

\begin{thm}[Equivalence for $\sigma>0$]\label{T:Ep} 
Let  $\sigma,\,\lambda>0,$ and assume that \eqref{CR} holds true. Given  $r\in[1,\infty)$,   set $\alpha:=(r-1)/r\in[0,1).$ 
 Then, the following formulations are equivalent:
\begin{itemize}
\item[(i)] The velocity formulation \eqref{eq:1}--\eqref{C1} for $u-c,v,P\in W^1_{r}(\0_\eta)\cap{\rm C}^{\alpha}(\ov\0_\eta) $, 
$\eta\in W^2_r(\R)$, and $\rho\in  W^1_{r}(\0_\eta)\cap{\rm C}^{\alpha}(\ov\0_\eta)$.\\[-2ex]
\item[(ii)] The stream function formulation \eqref{eq:P}--\eqref{C3} for $\psi\in W^2_{r}(\0_\eta)\cap{\rm C}^{1+\alpha}(\ov\0_\eta) $, 
$\eta\in W^2_r(\R)$, $\ov\rho\in  W^1_r((p_0,0))$, and $\beta\in L_r((p_0,0))$.\\[-2ex]
\item[(iii)] The height function formulation \eqref{eq:Q}--\eqref{C4} for $h\in W^2_{r}(\0)\cap{\rm C}^{1+\alpha}(\ov\0) $ with 
$\tr_0 h\in W^2_r(\R)$, $\ov\rho\in W^1_r((p_0,0))$, and $\beta\in L_r((p_0,0))$.
\end{itemize}
\end{thm}

The proof of Theorem \ref{T:Ep} and the corresponding result for $\sigma=0$ are presented at the end of this section.
It is worthwhile to add the following remarks.
\begin{rem}\label{R:1}
\begin{itemize}
\item[(a)] Given $r\geq 1$, the H\"older coefficient $\alpha:=(r-1)/r\in[0,1) $ corresponds to the one-dimensional Sobolev embedding 
$W^1_r(\R)\hookrightarrow {\rm C}^\alpha(\R).$
It is worthwhile to note that $W^1_r(\R^2)$ is not embedded in a space of continuous functions if $r\in[1,2]$. 
Therefore, the H\"older regularity required above is not implied by the Sobolev regularity.\\[-2ex]
\item[(b)] All function  spaces in    Theorem \ref{T:Ep} consist only of functions that are
 $\lambda$-periodic with respect to $x$ and $q$, respectively. \\[-2ex]
\item[(c)] The symbol $\tr_0$ stands for the trace operator with respect to the boundary component $p=0$ of $\0=\R\times(p_0,0),$ that is $\tr_0 h(q)=h(q,0),$ $q\in\R$, for
$h\in {\rm C}(\ov\0).$ \\[-2ex]
\item[(d)] Let $\Omega\subset \R^n$ with $n\ge 1$ be open. In the proof of Theorem \ref{T:Ep} (and also later on) we make use of the following properties
\begin{align}
\bullet \quad &\text{$\partial (u v) = u \partial v + v \partial u$ \  in $\mathcal{D}'(\Omega)$\; for  $u, v \in W^1_{1,loc}(\Omega)$ with 
$uv, u \partial v + v \partial u \in L_{1,loc}(\Omega);$}\label{PR1}\\[1ex]
\bullet \quad&\text{$W^1_r(\Omega)\cap {\rm BC}^\alpha(\Omega)$ is an algebra;} \label{PR2}\\[1ex]
\bullet \quad&\begin{minipage}{12cm}
If $f\in W^1_{1,loc}(\0)\cap{\rm BC}(\0)$ has weak derivatives $f_i\in{\rm BC}(\0),$ $1\leq i\leq n$, then $f\in {\rm BC}^1(\0)$.
\end{minipage} \label{PR3} 
\end{align}  
The properties  \eqref{PR1} and \eqref{PR3} are classical results, while  \eqref{PR2} is a direct consequence of~\eqref{PR1}. 
\end{itemize}
\end{rem}

 \pagebreak

\noindent{\bf Local bifurcation.}   The main issue of this paper is the   local bifurcation result stated below.
Under the natural assumptions \eqref{CR} and \eqref{RES1} on the fluid density and the following restrictions on the
physical quantities\footnote{The relations \eqref{RES*} are satisfied for example if $d$ is small compared to $|p_0|$.
 Furthermore, if $\beta=0=\ov\h'$, then $\mu_*=0$ and the first condition in \eqref{RES*} is trivially satisfied.}
 \begin{align}\label{RES*}
d+\frac{p_0}{\Big(\mu_*-2\,\underset{[p_0,0]}\min B\Big)^{1/2}}<0
\qquad\text{and}\qquad\frac{gd^{3} \ov\rho(p_0)|p_0|}{\Big[p_0^2-\Big(\mu_*-2\underset{[p_0,0]}\min B\Big)d^2\Big]^{3/2}}\leq \frac{x_*}{2},
\end{align}
 where  $x_*\approx1.9368$ is the  positive solution to $e^x-x=5$,
\begin{align}\label{Beta}
\mu_*:=2\Big(gd\|\ov\rho'\|_{L_1((p_0,0))}+\max_{[p_0,0]} B\Big),\qquad\text{and}\qquad B(p):=\int_{p_0}^p\beta(s)\, ds,\quad p\in[p_0,0],
\end{align}
we prove that the water wave problem \eqref{eq:1}--\eqref{C1} possesses, for each $\lambda>0$, at least one laminar flow solution with flat streamlines. 
Besides, a critical wavelength $\lambda_*>0$ is identified such that  \eqref{eq:1}--\eqref{C1} has also other solutions with non-flat wave surface and with wavelength
 close to $\lambda_*$. More precisely, the following result holds true.

\begin{thm}\label{MT1} Let $\sigma,\, d,\, -p_0\in (0,\infty),$ $r\in(1,\infty)$, and $\alpha:=(r-1)/r\in(0,1) $  be given.
Assume further that  $\ov\rho\in W^1_r((p_0,0))$ and $\beta\in L_r((p_0,0))$ satisfy \eqref{CR}, \eqref{RES1}, and \eqref{RES*}.
 Then there exists a local bifurcation curve  
 \[
\mathcal{C}=\{(\lambda(s), u(s)-c,v(s),P(s),\rho(s),\eta(s))\,:\, s\in(-\e,\e)\},
 \]
  where $\e>0$ is small, having  the following properties:
    \begin{itemize}
    \item[(i)] $\lambda$ is smooth, $\lambda(s)>0$ for all $s>0,$ and 
    \[
\lambda(s)=\lambda_*    +O(s) \quad \text{for $s\to0$},
    \]
    where $\lambda_*>0$ is defined in Proposition \ref{Prop:1}.\\[-2ex]
    \item[(ii)]  $(u(0)-c,v(0),P(0),\rho(0),\eta(0))$ is a strong solution  to \eqref{eq:1}--\eqref{C1} for each $\lambda>0$,  has  flat   streamlines,  streamline density $\ov\rho,$
     and Bernoulli function $\beta$.\\[-2ex]
    \item[(iii)] Given $ s\in(-\e,\e)\setminus\{0\}$, $(u(s)-c,v(s),P(s),\rho(s),\eta(s))$ is a strong  solution  to \eqref{eq:1}--\eqref{C1} with minimal period $\lambda(s),$ 
    streamline density $\ov\rho$, and Bernoulli function $\beta$. 
    Moreover, the wave profile   has precisely one crest and one
trough per period, is symmetric with respect to crest and trough lines, and  is strictly monotone between
crest and trough.\\[-2ex]
\item[(iv)]  The wave profile and all other streamlines are real-analytic graphs.
    \end{itemize}
 \end{thm}
 
 \begin{rem}\label{R:-2} 
 \begin{itemize}
 \item[(a)] We point out that we do not impose any   restrictions on the value of $\sigma>0,$ cf. \eqref{RES*}. 
 Nevertheless, the critical wavelength $\lambda_*$ depends in 
 an intricate way on $\sigma$. 
 \item[(b)] The regularity of the parametrization of $\mathcal{C}$ and the asymptotic behavior of $\eta(s)$ as $s\to0$ 
 are specified in  the proof of   Theorem \ref{MT1} at the end of Section  \ref{Sec:3}.
 \item[(c)] The strong solution $(u(0)-c,v(0),P(0),\rho(0),\eta(0))$ to~\eqref{eq:1}--\eqref{C1} found in (ii) is called laminar flow solution.
  Its existence is established in Proposition \ref{P:LS}. 
  \item[(d)] The limiting case $r=1$  remains open in the context of Theorem \ref{MT1} and  Theorem \ref{T:RR}.   
 \end{itemize}
 \end{rem}

  We conclude this section by proving the equivalence of the three formulations in the setting of strong solutions introduced above.

\begin{proof}[Proof of Theorem \ref{T:Ep}] We start with the implication (i)$\implies$(ii).
Let $(u-c,v,P,\rho,\eta)$ be a solution to \eqref{eq:1}--\eqref{C1}.
In virtue of  \eqref{CR} and the weak chain rule \cite[Lemma 7.5]{GT01} it follows that ${\sqrt\rho\in W^1_r(\0_\eta)\cap{\rm C}^{\alpha}(\ov\0_\eta)}$.
 Relation \eqref{PR2} then yields 
\[
\text{$U:=\sqrt\rho(u-c),\, V:=\sqrt\h v\in W^1_{r}(\0_\eta)\cap {\rm C}^{\alpha}(\ov\0_\eta)$.}
\]
We note that the relations $\eqref{eq:1}_3-\eqref{eq:1}_4$ imply 
\begin{align}\label{DIV'}
 \text{$U_x+V_y=0$ \quad in\quad $L_r(\0_\eta)$.} 
\end{align}
For $(x,y)\in\ov\0_\eta$ we now define
\begin{equation*}
	\psi(x,y):=-p_0+\int_{-d}^y U(x,s)\, ds, 
\end{equation*}
where $p_0<0$ is a  constant to be fixed below. 
It is obvious that $\psi$ is continuously differentiable with respect to $y$ with $\psi_y=U.$
Moreover, making use of Fubini's theorem, the generalized {Gau\ss} theorem in \cite[Appendix A 8.8]{Alt16}, and the 
relations \eqref{DIV'} and $\eqref{eq:1}_7$, we find for $\xi\in {\rm C}^\infty_0(\0_\eta)$ that
\begin{align*}
\int_{\0_\eta} \psi(x,y) \xi_x(x,y)\, d(x,y)&=\int_{\0_\eta}U(x,s) \Big(\int_s^{\eta(x)} \xi_x(x,y)\, dy\Big)\, d(x,s)
=\int_{\0_\eta}U(x,s)  \phi_x(x,s)\, d(x,s)\\[1ex]
&=-\int_{\0_\eta}U_x(x,s) \phi(x,s)\, d(x,s)=\int_{\0_\eta}V_s(x,s)  \phi(x,s)\, d(x,s)\\[1ex]
&=-\int_{\0_\eta}V(x,s) \phi_s(x,s)\, d(x,s) = \int_{\0_\eta}V(x,s)\xi(x,s)\, d(x,s),
\end{align*}
where $\phi\in {\rm C}^1(\ov\0_\eta)$ is defined by the formula
\[
\phi(x,s):=\int_s^{\eta(x)} \xi(x,y)\, dy,\qquad (x,y)\in \ov\0_\eta.
\]
Thus, $\nabla \psi=(-V,U)$ and since these weak derivatives belong to $W^1_{r}(\0_\eta)$,  we conclude that  ${\psi\in W^2_{r}(\0_\eta).}$
Moreover,  \eqref{PR3} implies that $\psi\in {\rm C}^{1+\alpha}(\ov\0_\eta)$.

The relation  \eqref{C2} is clearly satisfied in view of \eqref{C1}.
Since $\psi$ is constant on the fluid bed and by $\eqref{eq:1}_6$ also on the free surface $y=\eta(x)$, we infer from \eqref{C2}
that we may chose the negative constant $p_0$ such that $\psi=0$ on the free surface. 
 
 It is easy to see now that the mapping $\HH$ defined in \eqref{HHH} satisfies   $\HH\in {\rm Diff}^{1+\alpha}(\0_\eta,\0)$, i.e. $\HH:\0_\eta\to\0$ is a $C^{1+\alpha}$-diffeomorphism, with
\[
\begin{pmatrix}
 \frac{\p q}{\p x}&\frac{\p q}{\p y}\\[1ex]
  \frac{\p p}{\p x}&\frac{\p p}{\p y} 
\end{pmatrix}
=\begin{pmatrix}
 1&0\\[1ex]
   V&-U 
\end{pmatrix}
\qquad\text{and}\qquad \begin{pmatrix}
 \frac{\p x}{\p q}&\frac{\p x}{\p p}\\[1ex]
  \frac{\p y}{\p q}&\frac{\p y}{\p p} 
\end{pmatrix}\circ \HH
=
\begin{pmatrix}
 1&0\\[1ex]
  \frac{V}{U}&-\frac{1}{U}
\end{pmatrix}.
\]
In view of $\rho\in W^1_r(\0_\eta)$ it follows that $\rho\circ \HH^{-1}\in W^1_r(\0)$ with
\[
\p_q(\rho\circ \HH^{-1})\circ\HH= \rho_x+\frac{V}{U}\rho_y =0,
\]
cf. $\eqref{eq:1}_3.$
Consequently, there exists $\ov \rho\in L_r((p_0,0))$ with $\rho\circ \HH^{-1}=\ov\rho.$
Moreover, it  actually holds that $\ov\rho\in    W^1_r((p_0,0))$ with weak derivative $\ov\rho'=-(\rho_y/U)\circ \HH^{-1}.$
  
We now consider the expression
\[
E:=P+\frac{U^2+V^2}{2}+g\rho y
\]
which defines a function in $W^1_r(\0_\eta)\cap {\rm C}^\alpha(\ov\0_\eta)$.
Hence, $E\circ\HH^{-1}\in W^1_r(\0)$ and
\begin{align*}
 \p_q(E\circ \HH^{-1})\circ \HH&= E_x+\frac{V}{U}E_y \\[1ex]
 &=\big( \rho (u-c)u_x+\rho vu_y+P_x\big)+\frac{V}{U}\big( \rho (u-c)v_x+\rho vv_y+P_y+g\rho\big)\\[1ex]
 &\hspace{0.424cm}+((u-c)\rho_x+v\rho_y) \frac{E-P}{\sqrt{\rho}U}\quad \text{in $L_r(\0_\eta)$}.
\end{align*}
Appealing to $\eqref{eq:1}_1-\eqref{eq:1}_3$, it follows that  $\p_q(E\circ \HH^{-1})=0.$
This relation has at least two implications. 
Firstly, $E$ is constant at the wave surface, which implies the existence of a constant $Q$ such that
\begin{align}\label{CE}
 E=P +\frac{|\nabla\psi|^2}{2}+g\rho y=\frac{Q}{2} \quad \text{on \quad $ y=\eta(x) $}.
\end{align}
Secondly, in view of $\p_q(\p_p(E\circ \HH^{-1}))=0$, we may conclude   there exists a function ${\beta\in L_r((p_0,0))}$ such that 
$-\p_p(E\circ \HH^{-1})=\beta.$
The relation \eqref{CE} together with  $\eqref{eq:1}_5$ and $\eqref{eq:1}_8$, 
shows that $\eqref{eq:P}_4$ holds true with $Q$ as defined in \eqref{C3}.
Finally, since $\Delta\psi=U_y-V_x$,
$\eqref{eq:1}_2-\eqref{eq:1}_3$ lead us to
 \begin{align*}
 \beta\circ \HH&=\frac{1}{U}E_y=\Delta \psi-gy\ov\rho'\circ \HH,
\end{align*}
which is the semilinear elliptic equation in \eqref{eq:P}. This completes this first step of the proof.\medskip

We now verify that (ii)$\implies$(iii).
Let thus $(\psi,\eta)$ be a solution to \eqref{eq:P}--\eqref{C3} and let $h$ be the height function introduced in $\eqref{HF}$.
Then, it follows that $h\in {\rm C}^{1+\alpha}(\ov\0)$ with
\[
h_q=-\frac{\psi_x}{\psi_y}\circ\HH^{-1}\qquad\text{and}\qquad h_p:=-\frac{1}{\psi_y}\circ\HH^{-1}.
\]
With regard to \cite[Lemma 7.5]{GT01},  property \eqref{C2} shows that $1/\psi_y\in W^1_r(\0_\eta)\cap {\rm C}^{ \alpha}(\ov\0_\eta),$
and the algebra property \eqref{PR2} leads us to the conclusion that $h_q,\, h_p\in W^1_r(\0)$,
hence $h\in W^2_{r}(\0)$. 
The relation $\eqref{eq:P}_3$ implies that  $h$ satisfies $\eqref{eq:Q}_2$, while \eqref{C4}
follows immediately from \eqref{C2}.
Moreover, since 
\begin{align}\label{Psis}
 \psi_{xx}   = \frac{h_p^2 h_{qq}-2h_qh_ph_{qp}+h_q^2h_{pp}}{h_p^3}\circ\HH\qquad\text{and}\qquad \psi_{yy} = \frac{ h_{pp}}{h_p^3}\circ\HH,
\end{align}
it follows from $\eqref{eq:P}_1$ that $h$ is a solution to $\eqref{eq:Q}_1$.
Let us also note that the equation $\eqref{eq:P}_2$ yields $\eta(q)=h(q,0)-d$, $q\in\R$, and therefore $\tr_0 h\in W^2_r(\R)$.
The boundary condition $\eqref{eq:Q}_3$ is a direct consequence of $\eqref{eq:P}_4.$
This completes the second step of the proof.\medskip
  
It remains to establish the implication (iii)$\implies$(i).
To begin we first define $\eta:=h(\,\cdot\,,0)-d$. Then $\eta\in W^2_r(\R)$ and
integrating $\eqref{eq:Q}_3$ over one period of the wave we find that $\eqref{eq:1}_8$ is satisfied.
Besides, \eqref{C4} yields that $\eta(x)+d= h(x,0)>0$ for all $x\in\R$. 
We now associate to $\eta$ the corresponding velocity, pressure, and density distribution.
To this end we let $\Phi:\R^2\times[2p_0,0]\to\R$ be the function defined by
\[
\Phi(x,y,p):=\left\{
\begin{array}{lll}
y+d-h(x,p)&,& p\in[p_0,0],\\[1ex]
y+d+h(x,2p_0-p)&,& p\in[2p_0,p_0].
\end{array}
\right.  
  \]
Then $\Phi\in {\rm C}^{1+\alpha}(\R^2\times[0,2p_0])$, $\Phi(x,-d,p_0)=0$ for all $x\in\R,$ and $\Phi_p\leq -\inf_\0h_p<0$.
For fixed, but arbitrary $x\in\R$, the implicit function theorem yields the existence of a function $\psi(x,\,\cdot\,)$ which 
is continuously differentiable in $ [-d,-d+\e)$, 
for some $\e>0$, and satisfies $\psi(x,-d)=-p_0$ as well~as 
\[
h(x,-\psi(x,y))=y+d  \quad\text{for all \quad $y\in[-d,-d+\e)$.}
\] 
Because $\psi(x,\,\cdot\,)$ is strictly decreasing we can extend this function continuously in $-d+\e$ if
\[
\underset{y\to-d+\e}\lim\psi(x,y)<0.
\]
The implicit function theorem then enables us to even extend $\psi(x,\,\cdot\,)$ beyond $-d+\e$.
Hence,  $\psi(x,\,\cdot\,)$ has a maximal extension $\psi(x,\,\cdot\,)\in {\rm C}^1([-d,A(x)),\R)$ with $\psi(x,-d)=-p_0$ and $\psi(x,A(x))=0$.
In view of $A(x)+d=h(x,0)=\eta(x)+d$, we conclude that $\eta(x)=A(x)$. 
Therefore  $\psi:\ov\0_\eta\to\R$ and~\eqref{C2} is satisfied.
Moreover, the implicit function theorem yields that $\psi\in {\rm C}^{1+\alpha}(\0_\eta)$ with 
\[
\psi_x(x,y)= \frac{h_q(x,-\psi(x,y))}{h_p(x,-\psi(x,y))} \qquad\text{and}\qquad  \psi_y(x,y)=-\frac{1}{h_p(x,-\psi(x,y))}.
\] 
Since $h\in {\rm C}^{1+\alpha}(\ov\0),$ it now follows  that $\psi\in {\rm C}^{1+\alpha}(\ov \0_\eta)$ and the mapping $\HH$ 
defined in \eqref{HHH} obviously satisfies  $\HH\in {\rm Diff}^{1+\alpha}(\0_\eta,\0)$.
Since $h_q,\, h_p\in W^1_r(\0)\cap {\rm C}^\alpha(\0)$, we find in virtue of \eqref{PR1}--\eqref{PR2} and \cite[Lemma 7.5]{GT01} that
$\psi\in W^2_r(\0)$. Moreover, the derivatives $\psi_{xx} $ and $\psi_{yy} $ satisfy \eqref{Psis}. 
We now define $\rho,\, u-c,\,v :\ov\0_\eta\to\R$ by setting
 \[
\rho =\ov\rho\circ\HH,\qquad  \sqrt{\rho}(u-c) =\psi_y,\qquad   \sqrt{\rho} v =-\psi_x,
 \]
  and we let $ P:\ov\0_\eta\to\R$ be given by the relation
  \[
  P(x,y)=-\rho\frac{(u-c)^2+v^2}{2}(x,y)-g\rho(x,y) y-\int_0^{-\psi(x,y)}\beta(s)\, ds+\frac{Q}{2},
  \]
with $Q$ defined according to \eqref{C3}.
Then $\rho\in W^1_r(\0)\cap {\rm C}^{\alpha}(\ov\0)$ and, recalling \eqref{CR}, we may argue as above to conclude that $u$, $ v$, and $P$ belong to $ W^1_r(\0)\cap {\rm C}^{\alpha}(\ov\0).$
It is now a matter of direct computation to see that all the equations of \eqref{eq:1}--\eqref{C1} are satisfied.
This completes the proof.
\end{proof}

It follows  from the proof of Theorem \ref{T:Ep}
that, when neglecting surface tension effects, the following equivalence result holds.

\begin{thm}[Equivalence for $\sigma=0$]\label{T:E0} 
Let  ,$\sigma=0$, $\lambda>0$, and assume that \eqref{CR} holds true.
 Given  $r\in[1,\infty)$,  set $\alpha:=(r-1)/r\in[0,1).$
 Then, the following formulations are equivalent:
\begin{itemize}
	\item[(i)] The velocity formulation \eqref{eq:1}--\eqref{C1} for $u-c,v,P\in W^1_{r}(\0_\eta)\cap{\rm C}^{\alpha}(\ov\0_\eta) $, 
		$\eta\in {\rm C}^{1+\alpha}(\R)$, and $\rho\in  W^1_{r}(\0_\eta)\cap{\rm C}^{\alpha}(\ov\0_\eta)$.
	\item[(ii)] The stream function formulation \eqref{eq:P}--\eqref{C3} for $\psi\in W^2_{r}(\0_\eta)\cap{\rm C}^{1+\alpha}(\ov\0_\eta) $, 
		$\eta\in {\rm C}^{1+\alpha}(\R)$, $\ov\rho\in  W^1_r((p_0,0))$, and $\beta\in L_r((p_0,0))$.
		\item[(iii)] The height function formulation \eqref{eq:Q}--\eqref{C4} for $h\in W^2_{r}(\0)\cap{\rm C}^{1+\alpha}(\ov\0) $, $\beta\in L_r((p_0,0))$, 
		and $\ov\rho\in W^1_r((p_0,0))$.
\end{itemize}
\end{thm}

\section{A weak setting for Dubreil-Jacotin's formulation}\label{Sec:2}
In this section we seek  solutions to problem \eqref{eq:Q}-\eqref{C4} under the general assumptions that  
\begin{align}\label{L:LLL}
r\in(1,\infty),\qquad \ov\rho\in W^1_r((p_0,0)), \qquad \text{and}\qquad \beta\in L_r((p_0,0)),
\end{align}
where $\ov\rho$ and $\beta$ are arbitrary but  fixed.
Moreover, we restrict to the setting of   stably stratified flows defined by  \eqref{CR} and \eqref{RES1}.
The reason for studying the height function
formulation is twofold. 
Firstly, the equations have a single unknown, the  height function  $h$,  and secondly, the Bernoulli function $\beta$ and the
 streamline density  $\ov\rho$ appear as coefficients in the equations.

 Since we aim  to formulate \eqref{eq:Q} as a bifurcation problem and to use the wavelength $\lambda$ as   bifurcation parameter we let 
\begin{equation}\label{Scal}
\wt h(q,p):= h(\lambda q,p),\qquad (q,p)\in\ov\0.
\end{equation}
Then $\wt h$  is $1$--periodic\footnote{Hereinafter all function spaces consist of functions which are $1$--periodic with respect 
to $q$   (provided that they depend on the variable $q$).} and \eqref{eq:Q} may be rewritten (after dropping  tildes) as
\begin{equation}\label{eq:Q'}
\left\{
\begin{array}{rrllll}
(\lambda^2+h_q^2)h_{pp}-2h_qh_ph_{pq} +h_p^2h_{qq}-\lambda^2[g \ov\h' (h-d)+\beta]h_p^3&=&0 & \text{in $\0$},\\[1ex]
h&=&0 & \text{on } p=p_0,\\
 \lambda^2+h_q^2 +h_p^2\Big[2\lambda^2g \ov\h(0)(h-d)-\displaystyle 
 \frac{2\sigma\lambda^3 h_{qq}}{(\lambda^2+h_q^2)^{3/2}}- \int_{0}^1\frac{\lambda^2+h_q^2}{h_p^2}(q,0)\, dq\Big]&=&0 & \text{on  $ p=0$},
\end{array}
\right.
\end{equation}
while \eqref{C4}  remains unchanged.
 Now not only $h$ is unknown in \eqref{eq:Q'} but also the wavelength $\lambda$.

In order to determine strong solutions to~\eqref{eq:Q'} and~\eqref{C4} as defined in Theorem \ref{T:Ep} (iii), 
we shall first find weak solutions to this problem and then improve their regularity.
We now introduce a proper notion of weak solutions.  

\begin{defn}\label{D:31} A function $h\in {\rm C}^1(\ov\0)$
is called {\em weak solution} to~\eqref{eq:Q'} and~\eqref{C4}  if $h$  satisfies \eqref{C4}, the equation\footnote{Recall that $B$ denotes the primitive of the 
Bernoulli function, cf. \eqref{Beta}.} 
\begin{align*}
 \Big(\frac{h_q}{h_p}\Big)_q - \Big(\frac{\lambda^2 + h_q^2}{2h_p^2} + \lambda^2 B + \lambda^2 g \ov\rho (h-d) \Big)_p +\lambda^2 g \ov\rho h_p=0\qquad\text{in $\mathcal{D}'(\0)$},
\end{align*}
and the boundary conditions
 \begin{align*}
 h=(1-\p_q^2)^{-1}\tr_0\left[h-\frac{(\lambda^2+h_q^2)^{3/2}}{2\sigma\lambda^3}\left(\frac{\lambda^2+h_q^2}{h_p^2}+2\lambda^2g \ov\h(h-d)-\int_{0}^1\frac{\lambda^2+h_q^2}{h_p^2}\, dq\right)\right]\quad\text{on $p=0$ }
\end{align*}
and 
\begin{align*}
 h=0\quad\text{on $p=p_0$.}
\end{align*}
\end{defn}
In Definition \ref{D:31}    we have made use of the fact that $$(1-\p_q^2):{\rm C}^2(\R)\to {\rm C}(\R)$$ is an isomorphism.
\medskip

\paragraph{\bf Laminar flow solutions} 
In the remainder  of this section we show that, given any $\lambda>0$, the equations \eqref{eq:Q'} and~\eqref{C4}  have  at least one weak solution $H$ 
that depends only on the variable $p$. 
This solution is then easily seen to be a strong solution to~\eqref{eq:Q'} and~\eqref{C4} (similar as defined in Theorem \ref{T:Ep} (iii)).
This is the  laminar flow solution mentioned in Theorem \ref{MT1} (ii).

 Since the density is positive, it follows that  $H=H(p)$ is a weak solution to~\eqref{eq:Q'} and~\eqref{C4} if and only if $H'>0$ on $[p_0,0] $ and if $H$ solves the system
\begin{equation}\label{LFS}
 \left\{
\begin{array}{lllll}
 \displaystyle \Big(\frac{1}{H'^2}\Big)'&=& -2[g \ov\h' (H-d)+\beta]&\text{in $\mathcal{D}'((p_0,0))$},\\[2ex]
 H(0)&=&d,\\[1ex]
 H(p_0)&=&0.
\end{array}
\right.
\end{equation}
We emphasize that the wavelength parameter does not appear in \eqref{LFS}.
Taking into account that $H\in {\rm C}^1([p_0,0])$, we get that additionally $H\in W^2_r((p_0,0))$. 
Moreover, setting $\mu:=(H'(p_0))^{-2}>0$, the function $H$ satisfies the fixed point equation
\begin{align}\label{FPE}
 H(p)=\int_{p_0}^p\Big(\mu-2\int_{p_0}^r [g\ov\rho'(s)(H(s)-d)+\beta(s)]\, ds\Big)^{-1/2}\, dr,\qquad p\in[p_0,0].
\end{align}
Our goal is to show, by means of a shooting argument, that there exists a $\mu>0$ such that \eqref{FPE} has a solution which satisfies additionally $H(0)=d$.
 This solution then   also solves \eqref{LFS}.
 Let  $\mu_*\geq0$  and $B$ be as defined in \eqref{Beta}.
As a first step we prove below that the fixed point equation \eqref{FPE} has a unique nonnegative  solution  $H=H(\cdot;\mu)$
for any $\mu>\mu_*$.

\begin{prop}\label{T:todo}
Given $\mu>\mu_*$, there is  a unique  solution  $H=H(\cdot;\mu)\in   W^2_r((p_0,0))$ to the fixed point equation \eqref{FPE}. It further holds that  $H'>0$ in $[p_0,0]$. 
 \end{prop}
\begin{proof} Let $\mu>\mu_*$ be fixed.
Given $p_1\in(p_0,0]$    and $H\in{\rm C}([p_0,p_1],[0,\infty))$, we define
\[
T_1H(p):=\int_{p_0}^p\Big(\mu-2\int_{p_0}^r [g\ov\rho'(s)(H(s)-d)+\beta(s)]\, ds\Big)^{-1/2}\, dr,\qquad p\in[p_0,p_1].
\]
We now show that $T_1$ is a self-map. 
Indeed, recalling that $\ov \h'\leq 0$, it holds that
\begin{equation}\label{Estq}
\begin{aligned}
\mu-2\int_{p_0}^r [g\ov\rho'(s)(H(s)-d)+\beta(s)]\, ds
&=\mu-2B(r)-2\int\limits_{[H\leq d]}  g\ov\rho'(s)(H(s)-d) \, ds\\[1ex]
&\hspace{0,425cm}-2\int\limits_{[H> d]} g\ov\rho'(s)(H(s)-d)\, ds\\[1ex]
&\geq\mu-2\max_{[p_0,0]} B-2\int_{p_0}^{p_1}   gd|\ov\rho'(s)| \, ds\\[1ex]
&\geq \mu-\mu_*>0, \qquad r\in[p_0,p_1],
\end{aligned}
\end{equation}
and consequently  $T_1H \in{\rm C}([p_0,p_1],[0,\infty)).$ 

If $p_1$ is sufficiently close to $p_0$, then $T_1$ is a contraction.
Indeed,  given $H,\, K\in{\rm C}([p_0,p_1],[0,\infty))$ and $p\in[p_0,p_1],$ \eqref{Estq} shows that
\begin{align*}
|T_1H(p)-T_1K(p)|&\leq \frac{g(p_1-p_0)\|\ov\h'\|_{L_1((p_0,0))}}{(\mu-\mu_*)^{3/2}}\|H-K\|_{{\rm C}([p_0,p_1])}\leq \frac{1}{2}\|H-K\|_{{\rm C}([p_0,p_1])}
\end{align*}
if 
\begin{align}\label{R2}
p_1\leq p_{1,\mu} :=\min\Big\{0,\,p_0+\frac{(\mu-\mu_*)^{3/2}}{2g(\|\ov\h'\|_{L_1((p_0,0))}+1)}\Big\}.
\end{align}
Observing that ${\rm C}([p_0,p_1],[0,\infty))$ is a complete metric space, the Banach contraction principle yields the existence and uniqueness of a  nonnegative solution  $H_1\in W^2_r((p_0,p_{1,\mu}))$ to \eqref{FPE}.
We next prove  that as long as the right endpoint of the interval of existence does not reach $0$, we may extend it to the right by 
the amount of
$$\frac{(\mu-\mu_*)^{3/2}}{2g(\|\ov\h'\|_{L_1((p_0,0))}+1)}.$$
Indeed, assume that  $p_{1,\mu}<0$.
Given $p_2\in (p_{1,\mu},0]$ and $H\in C([p_{1,\mu},p_2],[0,\infty))$ we set
 \[
T_2H(p):=H_1(p_{1,\mu})+\int_{p_{1,\mu}}^p\Big(c_\mu-2\int_{p_{1,\mu}}^r [g\ov\rho'(s)(H(s)-d)+\beta(s)]\, ds\Big)^{-1/2}\, dr, \qquad p\in[p_{1,\mu},p_2],
\]
and 
\[
c_\mu:=(H_1'(p_{1,\mu}))^{-2}= \mu-2\int_{p_0}^{p_{1,\mu}} [g\ov\rho'(s)(H_1(s)-d)+\beta(s)]\, ds >0.
\]
The same arguments as above yield
\[
c_\mu-2\int_{p_{1,\mu}}^r [g\ov\rho'(s)(H(s)-d)+\beta(s)]\, ds\geq\mu-\mu_*>0,\qquad p\in[p_{1,\mu},p_2],
\]
hence $T_2H\in C([p_{1,\mu},p_2],[0,\infty))$ and $(T_2H)^{(k)}(p_{1,\mu}) =H_1^{(k)}(p_{1,\mu})$ for $k\in\{0,\,1\}.$
 Furthermore,  given  $H,\, K\in C([p_{1,\mu},p_2],[0,\infty))$ and $p\in[p_{1,\mu},p_2]$, it holds that 
\begin{align*}
&|T_2H(p)-T_2K(p)|\leq \frac{g(p_2-p_{1,\mu})\|\ov\h'\|_{L_1((p_0,0))}}{(\mu-\mu_*)^{3/2}}\|H-K\|_{C([p_{1,\mu},p_2])}\leq \frac{1}{2}\|H-K\|_{C([p_{1,\mu},p_2])}
\end{align*}
provided that 
\begin{align*} 
p_2\leq p_{2,\mu} :=\min\Big\{0,\,p_{1,\mu}+\frac{(\mu-\mu_*)^{3/2}}{2g(\|\ov\h'\|_{L_1((p_0,0))}+1)}\Big\}.
\end{align*}
Hence,  $T_2$ possesses a fixed point $H_2\in W^2_r((p_{1,\mu},p_{2,\mu}))$. 
Thus, we may extend $H_1$ to a solution to \eqref{FPE} which lies in  $ W^2_r((0,p_{2,\mu}))$ and which equals $H_2$ on $(p_{1,\mu},p_{2,\mu})$.
Arguing in this way, if necessary, we may extend (in a finite number of steps) $H_1$ onto the whole interval $[p_0,0].$
The uniqueness claim is obvious.
\end{proof} 

We next show that the solution found in Proposition \ref{T:todo} depends smoothly on the parameter~$\mu$.

\begin{lemma} 
	\label{L:PK}
For any $\mu>\mu_*$, let $H(\,\cdot\,;\mu)$ denote the solution to the fixed point equation \eqref{FPE} as given by Proposition \ref{T:todo}. Then, the mapping
 \begin{equation}\label{Map:mu}
[\mu\mapsto H(\,\cdot\,;\mu)]:(\mu_*,\infty)\to C([p_0,0])
 \end{equation} is smooth.
\end{lemma}

\begin{proof}
We prove that, given  $\e>0$, the mapping \eqref{Map:mu} is smooth on $(\mu_*+2\e g\|\ov\h'\|_{L_1((p_0,0))},\infty).$
	This claim follows by  applying the implicit function theorem to the equation $\kF(H,\mu)=0$, where 
	$\kF: \mathcal{U}_\e \times(\mu_*+2\e g\|\ov\h'\|_{L_1((p_0,0))},\infty)\subset {\rm C}([p_0,0])\times\R  \to {\rm C}([p_0,0])$ is defined by
	\begin{align*}
		\kF(H,\mu)(p):= H(p) -  \int_{p_0}^p\Big(\mu-2\int_{p_0}^r [g\ov\rho'(s)(H(s)-d)+\beta(s)]\, ds\Big)^{-1/2}\, dr, \qquad p\in [p_0,0].
	\end{align*}
	Here
	\[
	\mathcal{U}_\e:=\{H\in  {\rm C}([p_0,0])\,:\, H>-\e\}
	\]
	is an open subset of ${\rm C}([p_0,0])$.
	 Arguing as in the derivation of \eqref{Estq} it can be seen that the operator $\kF$ is well-defined.
	Moreover, $\kF$ is  smooth. 
	The partial derivative $\partial_H\kF(H,\mu)[\wt  H] $ of $\kF$ with respect to $H$ at a given point $(H,\mu)\in \mathcal{U}_\e \times(\mu_*+2\e g\|\ov\h'\|_{L_1((p_0,0))},\infty)$ 
	can be expressed as
	\begin{align*}
		\partial_H\kF(H,\mu): {\rm C}([p_0,0]) \to {\rm C}([p_0,0]),\quad \wt H \mapsto  \wt H - \cK[\wt H], 
	\end{align*}
	where the operator $\cK:{\rm C}([p_0,0])\to{\rm C}([p_0,0])$ is given by
	\begin{align*}
		\cK[\wt H](p)
		= g \int_{p_0}^{p}\Big(\mu-2\int_{p_0}^r [g\ov\rho'(H-d)+\beta]\, ds\Big)^{-3/2} \Big(\int_{p_0}^r\ov\rho'\wt H\, ds \Big)\, dr, \qquad p\in [p_0,0].
	\end{align*}
	One can easily verify that $\cK$ actually maps continuously into $W^2_r((p_0,0))$. 
	Since the embedding of $W^2_r((p_0,0))$ into ${\rm C}([p_0,0])$ is compact it follows that $\cK$ is a compact operator. 
	Hence, $\partial_H\kF(H,\mu)$ is a compact perturbation of the identity. 
	Using the Riesz--Schauder theorem, we can conclude that $\partial_H\kF(H,\mu)$ is a Fredholm operator of index zero. 
	Furthermore, $\partial_H\kF(H,\mu)[\wt H]=0$ if and only if $\wt H = \cK[\wt H]$. In this case, $\wt H$ satisfies the inequality
	\begin{align*}
		|\wt H(p)|&		\le g \int_{p_0}^p (\mu-\mu_*-2\e g\|\ov\h'\|_{L_1((p_0,0))})^{-3/2} \int_{p_0}^r |\ov\rho'(s)|\cdot |\wt H(s)| \,ds \,dr\\[1ex]
		&\le g |p_0| (\mu-\mu_*-2\e g\|\ov\h'\|_{L_1((p_0,0))})^{-3/2} \int_{p_0}^p |\ov\rho'(s)|\cdot |\wt H(s)| \,ds 
	\end{align*} 
	for all $p\in[p_0,0]$. Hence, applying Gronwall's lemma, we obtain that $\wt H = 0$ on $[p_0,0]$. 
	This means that $ \ker\big(\partial_H\kF(H,\mu)\big)=\{0\}$ and thus, according to Fredholm's alternative, 
	$\partial_H\kF(H,\mu)$ is an isomorphism from ${\rm C}([p_0,0])$ to ${\rm C}([p_0,0])$.
	
	Due to the construction of $\kF$ it holds that
	\begin{align*}
		\text{$\kF(H(\,\cdot\,;\mu),\mu)=0\qquad$ for all $\mu\in(\mu_*+2\e g\|\ov\h'\|_{L_1((p_0,0))},\infty)$.} 
	\end{align*}
	Since the solution to \eqref{FPE} is unique also in $\mathcal{U}_\e$  if $\mu>\mu_*+2\e g\|\ov\h'\|_{L_1((p_0,0))}$ (this follows in a similar way as in Proposition \ref{T:todo}), the implicit function theorem implies that 
	$[\mu \mapsto H(\,\cdot\,;\mu)]$ is smooth in $(\mu_*+2\e g\|\ov\h'\|_{L_1((p_0,0))},\infty)$ for all $\e>0$.
	\end{proof}

In order to establish the existence of a solution to \eqref{LFS} it remains to prove that there is a $\mu > \mu_*$ such that 
$H(0;\mu)=d$. To this end   first note that  \eqref{Estq}  yields
\begin{equation}\label{eq:est_H(p,mu)}
H(p;\mu)\leq  \frac{|p_0|}{(\mu-\mu_*)^{1/2}} \qquad\text{for $p\in[p_0, 0]$}
\end{equation}
and hence $H(0;\mu)<d$ for sufficiently large $\mu$.
 If we additionally  show  that $H(0;\mu) > d$ for $\mu$ close to $\mu_*$, then the continuous dependence of $H$ on
$\mu$ (Lemma \ref{L:PK}) implies that $H(0;\mu)=d$ for some $\mu$. 
It turns out that this requires additional restrictions on the physical quantities.
The following example illustrates the approach in the simplified case of constant density.

\begin{ex}\label{Ex:1} Assume that $\ov\rho'=0$ and that $\beta(p)=C(p-p_0)^{-1/2}$.
Since the density is constant, the  function $\beta$ has to be interpreted as  the vorticity function of the flow, cf. \cite{CS04}. 
Unbounded vorticity functions are of relevance for example for flows in channels \cite{B62}. 
By Proposition \ref{T:todo}, the equation \eqref{FPE} possesses a unique solution $H=H(p;\mu)$  for each  $\mu>\mu_*=2C|p_0|^{1/2}$. 
The solution is  given by the explicit formula
\[
H(p;\mu)=\int_{p_0}^p\frac{1}{\sqrt{\mu-2 B(r)}}\, dr,\qquad p_0\leq p\leq 0.
\]
Since $H(0;\mu)< H(0,\mu_*) $ for $\mu>\mu_*$ and because of
\[
H(0;\mu_*)=\int_{p_0}^0\frac{1}{\sqrt{2B(0)-2 B(r)}}\, dr=\frac{1}{2\sqrt{C}}\int^0_{p_0}\frac{\sqrt{|p_0|^{1/2}+(r-p_0)^{1/2}}}{\sqrt{-r}}\, dr<\infty,
\]
the condition $H(0;\mu^*)>d$ has   to be  imposed, otherwise $H(0;\mu)<d$ for all $\mu>\mu_* $ and \eqref{LFS} has no solutions.
\end{ex}
\medskip

Let us now return  to the more involved setting of stratified waves addressed in this paper.
Given $\mu>\mu_*$, we define 
\[
p_\mu:= \min\big(\{0\}\cup H(\cdot\,;\mu)^{-1}(\{d\})\big),
\]  
Then recalling \eqref{eq:est_H(p,mu)},  we observe   that  $p_\mu=0$ and  $H(0;\mu)<d$ if $\mu $ is sufficiently large. 
The size condition  that we require reads  
\begin{align}\label{RES2}
d+\frac{p_0}{\Big(\mu_*-2\,\underset{[p_0,0]}\min B\Big)^{1/2}}<0.
\end{align}
Since $H(p;\mu)\in[0,d]$ for $p\in[p_0, p_\mu]$ and $\ov\h'\leq 0$, it follows that
\begin{align*}
 H(p_\mu;\mu)&=\int_{p_0}^{p_\mu} \Big(\mu-2\int_{p_0}^r [g\ov\rho'(s)(H(s)-d)+\beta(s)]\, ds\Big)^{-1/2} dr \geq \int_{p_0}^{p_\mu} \Big(\mu-2\int_{p_0}^r  \beta(s) \, ds\Big)^{-1/2}\, dr\\[1ex]
&\geq \frac{p_\mu-p_0}{\Big(\mu-2\,\underset{[p_0,0]}\min B\Big)^{1/2}}.
\end{align*}
In view of \eqref{RES2} we conclude (by a contradiction argument) that $p_\mu<0$ for all $\mu $ which are sufficiently close to $\mu_*.$
Hence
\[
H(0;\mu)>d \qquad\text{for  $\mu $   sufficiently close to $\mu_*.$}
\]
Together with Lemma \ref{L:PK} we conclude the following result establishing the existence of at least one strong laminar flow solution   to~\eqref{eq:Q'} and~\eqref{C4}.

\begin{prop}
	\label{P:LS}
	Assume that \eqref{RES2} is satisfied. 
	Then there is at least one solution $H\in W^2_r((p_0,0))$ to \eqref{LFS}. Moreover, $H'$ is a positive function.
\end{prop}
\begin{proof}
The proof follows from Proposition \ref{T:todo}, Lemma  \ref{L:PK}, and the discussion preceding Proposition~\ref{P:LS}.
\end{proof}

 \section{Local bifurcation}\label{Sec:3}
 
 In the following we assume that the density and the Bernoulli function satisfy  \eqref{CR}, \eqref{RES1}, and \eqref{L:LLL}, i.e., \begin{equation*}
 	\ov{\rho} \geq \rho_0,\, \ov{\rho}^{\prime} \leq 0
	\quad
	\text{in } [p_0,0],
	\quad
	\ov{\rho} \in W^1_r((p_0,0)),\ \beta \in L_r((p_0,0)),\quad r \in (1,\infty),
 \end{equation*}
 and that \eqref{RES2} holds true. 
 This guarantees, in particular, the existence of a laminar flow solution to~\eqref{eq:Q'} and~\eqref{C4}.
  
 The first goal of this section is to recast the weak formulation of~\eqref{eq:Q'} and~\eqref{C4} as an abstract bifurcation problem.
 For this goal only the H\"older regularity of the strong solutions to~\eqref{eq:Q'} and~\eqref{C4}  is needed.
  To proceed, we define the Banach space 
 \begin{align*}
 	\Y_1 := \{ \varphi + \partial_q\psi_1 + \partial_p \psi_2 \in \D'(\0) \,:\, \varphi\in L_\infty(\0), 
 	\, \psi_1,\,\psi_2\in {\rm C}^\alpha(\ov\0),\; \text{$\varphi,\psi_2$ are even, $\psi_1$ is odd in $q$}\}
 \end{align*}
 which is endowed with the norm 
 \begin{align*}
	\|u\|_{\Y_1} := \inf\{ \|\varphi\|_\infty + \|\psi_1\|_\alpha + \|\psi_2\|_\alpha \,:\,
	 u=\varphi + \partial_q\psi_1 + \partial_p \psi_2 ,\, \varphi\in L_\infty(\0), \, \psi_1,\,\psi_2\in {\rm C}^\alpha(\ov\0) \}
 \end{align*}
 and we set 
 \begin{align*}
 	\Y_2 &:= \{ \phi \in {\rm C}^{1+\alpha}(\R) \,:\, \text{$\phi$ is even} \},\\[1ex]
	\X &:= \{ h\in {\rm C}^{1+\alpha}(\ov\0) \,:\, \text{$h$ is even in $q$ and $h=0$ on $p=p_0$} \}.
 \end{align*}
 We recall that these Banach spaces consist only of  periodic distributions of period $1$.
 Furthermore, we fix a  laminar flow solution  $H \in W^2_r((p_0,0)) \hookrightarrow \X$ (as found in Proposition \ref{P:LS}) and we 
 let $\cO$ denote the open subset of $\X$ defined by
  \begin{align*}
 	\cO := \left\{ h\in\X \,:\, \underset{\overline{\Omega}}{\min}(h_p + H') >0 \right\}.
 \end{align*}
 
 The weak formulation of~\eqref{eq:Q'} and~\eqref{C4} can then be recast as the nonlinear and nonlocal equation 
 \begin{align}
 	\label{BP1}
 	\kF(\lambda,h) = 0,
 \end{align}
 where $\kF:=(\kF_1,\kF_2):(0,\infty)\times\cO\to \Y:=\Y_1\times\Y_2$ is given by
 \begin{align*}
 	\kF_1(\lambda,h) &:= \Big(\frac{h_q}{h_p+H'}\Big)_q - \Big(\frac{\lambda^2 + h_q^2}{2(h_p+H')^2} + \lambda^2 B + \lambda^2 g \ov\rho (h+H-d) \Big)_p +\lambda^2 g \ov\rho (h_p + H') ,\\[1ex]
 	\kF_2(\lambda,h) &:= \tr_0 h - (1-\p_q^2)^{-1}\tr_0\left[h-\frac{(\lambda^2+h_q^2)^{3/2}}{2\sigma\lambda^3} \left(\frac{\lambda^2+h_q^2}{(h_p+H')^2}+2\lambda^2g \ov\h h \right.\right.  \\[1ex]
 	&\hspace{7.6cm} \left.\left. -\int_{0}^1\frac{\lambda^2+h_q^2}{(h_p+H')^2}\, dq\right)\right].
 \end{align*}
 We regard the equation \eqref{BP1} as a bifurcation problem with   bifurcation parameter $\lambda$.
  Recalling that $H$ solves \eqref{LFS}, it holds  that 
 \begin{align}
 	\label{TS}
 	\kF(\lambda,0) = 0 \quad\text{for all $\lambda>0$}.
 \end{align}
 We note that $\kF$ is smooth with respect to its variables, that is 
  \begin{align}
 	\label{REG}
 	\kF\in {\rm C}^\infty((0,\infty)\times\cO,\Y).
 \end{align}
 The aim is to apply the Crandall--Rabinowitz theorem  \cite[Theorem~1.7]{CR71} on bifurcation from simple eigenvalues to \eqref{BP1} to determine 
 other solutions to \eqref{BP1} that depend on the variable $q$.
  To this end we need to determine $\lambda_*>0$ such that the partial Fr\'echet derivative $\partial_h\kF(\lambda_*,0)$ is a Fredholm operator of index zero with a one-dimensional kernel.
   A certain transversality condition also needs  to be satisfied, cf. Proposition \ref{Prop:2}.
   
   To begin, we observe that, given $\lambda>0$, the Fr\'echet derivative $\partial_h\kF(\lambda,0) = (L,T) \in \kL(\X,\Y)$ is expressed by
 \begin{align*}
 	L[h] &= \Big(\frac{h_q}{H'}\Big)_q + \lambda^2 \Big(\frac{ h_p}{H'^3}-g \ov\rho h \Big)_p +\lambda^2 g \ov\rho h_p \in \Y_1, \\[1ex]
 	T[h] &= \tr_0 h - (1-\p_q^2)^{-1}\tr_0\Big[h+\frac{\lambda^2}{\sigma} \Big(\frac{h_p}{H'^3} - g \ov\h h 
 	 	- \int_{0}^1 \frac{h_p}{H'^3}\, dq\Big)\Big]\in\Y_2, \qquad h\in\X.
 \end{align*}
 
\begin{lemma}\label{L:31}
	Given $\lambda>0,$ it holds that $\partial_h\kF(\lambda,0)$ is a Fredholm operator of index zero.
\end{lemma}

\begin{proof}
	According to \cite[Theorem 8.34]{GT01}, the mapping 
	\begin{align*}
		\Big[ h \mapsto \Big( \Big(\frac{h_q}{H'}\Big)_q + \lambda^2 \Big(\frac{ h_p}{H'^3} \Big)_p \,,\, \tr_0 h \Big) \Big] : \X\to\Y
	\end{align*}
	is an isomorphism (we recall that $H'$ is positive).
	 Since $(1-\partial_q^2)^{-1} \in \mathcal{L}({\rm C}  (\R),{\rm C}^2(\R))$, it follows that the operator
	\begin{align*}
		\left[ h \mapsto \Big( -\lambda^2 g (\ov\rho h)_p +\lambda^2 g \ov\rho h_p \,,\,  - (1-\p_q^2)^{-1}\tr_0\Big[h+\frac{\lambda^2}{\sigma} \Big(\frac{h_p}{H'^3} - g \ov\h h 
		- \int_{0}^1 \frac{h_p}{H'^3}\, dq\Big)\Big] \Big) \right] : \X\to\Y
	\end{align*}
	is compact. Therefore, the desired claim follows.
\end{proof}

We want to determine  special values $\lambda_*$ of the wavelength parameter such that additionally  the kernel of $\partial_h\kF(\lambda_*,0)$   is one-dimensional.
Let thus $w\in\X$ be such that $\partial_h\kF(\lambda,0)[w] = 0$. From $L[w] = 0$ we find that the Fourier coefficient 
\begin{align*}
	w_k(p) := \int_0^1 w(q,p) \cos(2k\pi q) \,dq,\qquad k\in\N,\, p\in[p_0,0],
\end{align*}
satisfies the equation
\begin{align*}
	\lambda^2\Big( \frac{1}{H'^3} w_k' -  g \ov\rho w_k \Big)' + \lambda^2g\ov\rho w_k' - \frac{(2k\pi)^2}{H'} w_k = 0 \qquad \text{in $L_r((p_0,0))$}.
\end{align*}
Furthermore, since $T[w]=0$, we find that
\begin{align*}
	\big( \lambda^2 g\ov\rho(0) + \sigma(2k\pi)^2 \big)w_k(0) = \frac{\lambda^2}{H'^3(0)} w_k'(0),\qquad k\ge 1,
\end{align*}
and
\begin{align*}
	w_0(0) = 0,
\end{align*}
respectively.
Finally, since $w\in\X$, it holds that
\begin{align*}
	w_k(p_0) = 0,\qquad k\in\N.
\end{align*}
 Hence, for $k=0$, we find that  $w_0\in W^2_r((p_0,0))$ solves the system
\begin{equation}\label{BN0} 
 \begin{aligned}
 \left\{
 \begin{array}{lll}
  ( a^3 w_0')' - g\ov\rho' w_0  = 0 \qquad  \text{in $L_r((p_0,0))$},\\[2ex]
 w_0(p_0)=w_0(0)=0,
 \end{array}
 \right.
 \end{aligned}
 \end{equation}
where $$a:=1/H'\in W^1_r((p_0,0)).$$
Furthermore, given $k\geq 1,$ the function $w_k\in W^2_r((p_0,0))$ solves the system
\begin{equation}\label{BN} 
 \begin{aligned}
 \left\{
 \begin{array}{lll}
\lambda^2( a^3w' )' - \lambda^2g\ov\rho' w - \vartheta a w = 0 \qquad \text{in $L_r((p_0,0))$},\\[2ex]
\big( \lambda^2 g\ov\rho(0) + \sigma\vartheta \big)w(0) = \lambda^2a^3(0) w'(0),\\[2ex]
 w(p_0)=0.
 \end{array}
 \right.
 \end{aligned}
 \end{equation}
with $\vartheta:=(2k\pi)^2$.

The remainder of this section is organized as follows. In the first paragraph we specify a condition under which
\eqref{BN0} has only the trivial solution $w_0=0$. 
The objective  of the second paragraph is twofold. On the one hand we 
determine special values $\lambda_*$ of $\lambda$ such that \eqref{BN} has a one-dimensional space of solutions for $k=1$, 
respectively $\vartheta=(2\pi)^2$. 
On the other hand we prove that \eqref{BN} admits only the trivial solution $w_k=0$ for 
$\vartheta\in\{(2k\pi)^2\,:\, k\geq2\}$. 
The third paragraph treats the above mentioned transversality condition and the last paragraph is devoted to the proof of Theorem \ref{MT1}. 
\medskip

\paragraph{\textbf{Conditions such that $w_0=0$.}}
We now show that, under some additional restrictions (cf. \eqref{C:the0}), the system \eqref{BN0} has only the trivial zero solution.
Equivalently formulated, we show that the elliptic operator  $[u\mapsto ( a^3 u')' - g\ov\rho'u]$, which is supplemented  by homogeneous Dirichlet boundary conditions, does not have zero as an eigenvalue.
We point out that the coefficient of $u$ is positive and it may be unbounded, while the coefficient $a^3$ is not explicitly determined.
 This is where the assumption $\eqref{RES*}_2$ becomes important, as it constitutes
an explicit relation in terms of $d$, $p_0$, $\beta$, and $\ov\rho$ which ensures that \eqref{C:the0} is satisfied, see Example~\ref{Bei:2} below. 
We also emphasize that in the constant density case the assertion of Lemma \ref{L:w0} follows via maximum principles, hence \eqref{C:the0} (or \eqref{RES3}) is then not needed.
Moreover,  the restriction \eqref{C:the0} (or its weaker version \eqref{RES3}) is also used at several places in the   paragraphs below in order to establish the existence of a bifurcation point.
In particular, for homogeneous irrotational waves  (that is $\ov\rho=\rho\in\R$ and $\beta=0$), see Remark \ref{R:CD}, the explicit condition \eqref{RES3} ensures, 
due to the fact that $x_*<2$, that  the dispersion relation \eqref{DR} has at least a solution.
Since $x_*\approx1.9368$, this  shows that  \eqref{RES3} is close to being optimal  with respect to this issue.

 \begin{lemma}\label{L:w0}
Let 
\begin{align}\label{AAA}
A(p):=\int_{p_0}^p  \frac{g\ov\rho}{a^3}(s)\, ds,\qquad p\in[p_0,0],
\end{align}
and assume that 
\begin{align}\label{C:the0}
 e^{2A(0)}-2A(0)\leq 5.
 \end{align}
 Then \eqref{BN0} has only the trivial solution $w_0=0$.
\end{lemma}
\begin{proof}
Let $w_0$ be a solution to \eqref{BN0}. 
Then, letting $z_0:=a^3w_0'-g\ov\rho w_0$ we may recast $\eqref{BN0}_1$  as a linear system of first order ODEs with continuous coefficients, namely
 \begin{equation}\label{SYS0}
\left\{
\begin{array}{lll}
w_0' = A' w_0+a^{-3}z_0,\\[1ex]
z_0' = -g\ov\rho A'w_0-A' z_0 .
\end{array}
\right.
\end{equation}
To obtain a contradiction, we assume that  $z_0(p_0)=a^3(p_0)w_0'(p_0)\neq0$ as the solution to \eqref{SYS0} that satisfies $(w_0,z_0)(p_0)=(0,0)$ is the trivial one. 
 Without loss of generality let $z_0(p_0)=:\alpha>0$. 
We next show that   $z_0>0$ in $[p_0,0]$ if \eqref{C:the0} holds true.
Since, by $\eqref{SYS0}_1$,
\begin{align}\label{L:efe}
w_0(p)e^{-A(p)}=\int_{p_0}^p\frac{z_0e^{-A}}{a^3}(s)\, ds,\qquad p\in[p_0,0],
\end{align}
this then contradicts the boundary condition $w_0(0)=0$ and the proof is complete. 

To prove that  $z_0$ is positive in $[p_0,0]$, we assume there exists  $p_1\in(p_0,0] $ with $z_0>0$ in $[p_0,p_1)$ and $z_0(p_1)=0$.
 The relation \eqref{L:efe} implies that  $w_0>0$ in $(p_0,p_1]$. 
 Invoking $\eqref{SYS0}_2$ it holds that
\begin{align}\label{EXC}
z_0(p)e^{A(p)}=\alpha-\int_{p_0}^p g\ov\rho(s)A'(s)e^{2A(s)}\int_{p_0}^s\frac{z_0(r)e^{-A(r)}}{a^3(r)}\, dr\, ds, \qquad p\in[p_0,0].
\end{align}
Since $A'$ is positive, we find that $z_0e^A$ is decreasing in $[p_0,p_1]$. 
  Consequently  
 \begin{align}\label{Bound10}
 0\leq z_0(p)< \alpha e^{-A(p)},\qquad p\in(p_0,p_1].  
 \end{align}
 Since $z_0(p_1)=0,$ the relation \eqref{EXC} yields
 \begin{align*}
 \alpha=\int_{p_0}^{p_1}g \ov\rho (s)A'(s)e^{2A(s)}
  \int_{p_0}^s \frac{z_0(r)e^{-A(r)}}{a^3(r)}\, dr\, ds,
 \end{align*}
 and using \eqref{Bound10} we arrive at
 \begin{align*}
 1<\int_{p_0}^{p_1}g\ov\rho(s)A'(s)e^{2A(s)}
  \int_{p_0}^s \frac{e^{ -2A(r)}}{a^3(r)}\, dr\, ds.
 \end{align*}
 Since $\ov\rho(s)\leq \ov\rho(r)$ for $p_0\leq r\leq s\leq 0$ and recalling the definition of $A$ we get in view of \eqref{C:the0} that
  \begin{align*}
 1&<\int_{p_0}^{p_1}A'(s)e^{2A(s)}
  \int_{p_0}^s A'(r)e^{ -2A(r)}\, dr\, ds=\frac{1}{2}\int_{p_0}^{p_1}A'(s)(e^{2A(s)} -1)\, ds\\[1ex]
  &\leq \frac{1}{4}\big(e^{2A(0)}-2A(0)-1\big)\leq 1,
 \end{align*}
 which is a contradiction.
Our assumption is thus false and the proof complete.
\end{proof}

We now provide  a quantitative condition which ensures that \eqref{C:the0} is satisfied.
 
\begin{ex}\label{Bei:2} Let $\mu_*$ be as defined in \eqref{RES*} and assume that \eqref{RES2} holds. 
If
\begin{align}\label{RES3}
\frac{gd^{3} \ov\rho(p_0)|p_0|}{\Big[p_0^2-\Big(\mu_*-2\underset{[p_0,0]}\min B\Big)d^2\Big]^{3/2}}\leq \frac{x_*}{2},
\end{align}
where $x_*\approx1.9368$ is the positive solution to
\begin{equation}\label{exx}
e^x-x=5,
\end{equation}
 then \eqref{C:the0} is satisfied.
\end{ex}
\begin{proof}
Recall that $a=1/H'>0$ in $[p_0,0]$ with $H$ being the solution to \eqref{LFS}   that we fixed earlier.
 Since $H(p_0)=0$ and $H(0)=d$, there exists $p_1\in[p_0,0]$ such that $H'(p_1)\leq d/|p_0|.$
Integration of $\eqref{LFS}_1$ over $[p_1,p],$ with $p\in[p_0,0]$ arbitrary, yields
\begin{align*}
a^2(p)&=a^2(p_1)-2\int_{p_1}^p[g\ov\rho'(s)(H(s)-d)+\beta(s)]\, ds\\[1ex]
&\geq \frac{p_0^2}{d^2}-2\int_{p_0}^0gd|\rho'(s)|\, ds-2B(p)+2B(p_1)\\[1ex]
&\geq \frac{p_0^2}{d^2}-\Big(\mu_*-2\min_{[p_0,0]} B\Big),
\end{align*}
the positivity of the constant on the right-hand side  of this inequality being equivalent to~\eqref{RES2}.
Consequently, 
\[
\frac{1}{a(p)}\leq\frac{d}{\sqrt{p_0^2-\Big(\mu_*-2\underset{[p_0,0]}\min B\Big)d^2}},\quad p\in[p_0,0],
\]
and, recalling that $\ov\rho\leq \ov\rho(p_0)$, we get
\begin{align}\label{kl*}
A(0)=\int_{p_0}^0\frac{g\ov\rho(s)}{a^3(s)}\, ds\leq g\ov\rho(p_0)\int_{p_0}^0\frac{1}{a^3(s)}\, ds\leq \frac{gd^3\ov\rho(p_0)|p_0|}{\Big[p_0^2-\Big(\mu_*-2\underset{[p_0,0]}\min B\Big)d^2\Big]^{3/2}}\leq\frac{x_*}{2}.
\end{align}
Hence \eqref{C:the0} holds true.
\end{proof}
 
Relation \eqref{RES3} provides an explicit condition which ensures that the system \eqref{BN0} has only the trivial solution  $w_0=0$.
 Consequently, for all $\lambda>0$, the kernel of  $\partial_h\kF(\lambda,0)$ does not contain functions that depend only on the variable $p$ (except  for the  zero function).
 We now address the second issue of determining $\lambda_*>0$ such that $\partial_h\kF(\lambda_*,0)$ has a one-dimensional kernel spanned by a function of the form $w_1(p)\cos(2\pi q)$, with 
 $w_1$ being (up to a multiplicative constant) the only  nontrivial solution to \eqref{BN}  when $\lambda=\lambda_*$ and $\vartheta\in\{(2k\pi)^2\,:\, k\in\N\}$.\medskip

\paragraph{\textbf{The system \eqref{BN}  with $\vartheta$ as a variable.}}
We seek  $\lambda>0$ such that    \eqref{BN} has a one-dimensional space of solutions   for $\vartheta=(2\pi)^2$ 
and only the trivial solution for $\vartheta>(2\pi)^2$.
To this end we first  determine the dimension of the space of solutions to \eqref{BN}. 
Given $\lambda\in (0,\infty)$ and $\vartheta\in\R$, we let $R_{\lambda,\vartheta}:W^2_{r,0}((p_0,0))\to L_r((p_0,0))\times\R$ denote the Sturm--Liouville type operator
\begin{align}\label{SLop}
R_{\lambda,\vartheta}[w]:=\begin{pmatrix}
\lambda^2( a^3w' )' - \lambda^2g\ov\rho' w - \vartheta a w\\[1ex]
 \lambda^2a^3(0) w'(0)-\big( \lambda^2 g\ov\rho(0) + \sigma\vartheta \big)w(0)
\end{pmatrix},
\end{align}
where we set
\begin{align*}
W^2_{r,0}((p_0,0))=\{w\in W^2_{r}((p_0,0))\,:\, w(p_0)=0\}.
\end{align*}
We associate to \eqref{BN} the initial value problems
\begin{equation}\label{SYS1}
\left\{
\begin{array}{lll}
\lambda^2( a^3\wt w_1' - g\ov\rho \wt w_1)' + \lambda^2g\ov\rho \wt w_1' - \vartheta a \wt w_1 = 0\qquad \text{in $L_r((p_0,0))$},\\[1ex]
\wt w_1(p_0)=0,\quad \wt w_1'(p_0)=1,
\end{array}
\right.
\end{equation}
and 
\begin{equation}\label{SYS2}
\left\{
\begin{array}{lll}
\lambda^2( a^3\wt w_2' - g\ov\rho \wt w_2)' + \lambda^2g\ov\rho \wt w_2' - \vartheta a \wt w_2 = 0\qquad \text{in $L_r((p_0,0))$},\\[1ex]
\wt w_2(0)=  \lambda^2a^3(0),\quad \wt w_2'(0)=\lambda^2 g\ov\rho(0) + \sigma\vartheta.
\end{array}
\right.
\end{equation}

\begin{lemma}
\label{LEM:KER}
Given $\lambda>0$  and $\vartheta\in\R$, the operator~$R_{\lambda,\vartheta}:W^2_{r,0}((p_0,0))\to L_r((p_0,0))\times\R$ defined in \eqref{SLop}
is a Fredholm operator of index zero and ${\rm dim} \ker R_{\lambda,\vartheta}\le 1$.
Furthermore,   ${\rm dim} \ker R_{\lambda,\vartheta}= 1$ if and only if  the solutions $\wt w_1$ and $\wt w_2$  to $\eqref{SYS1}$ and $\eqref{SYS2}$  
are linearly dependent. In this case it holds that
 $$\ker R_{\lambda,\vartheta}={\rm span}\{\wt w_1\}={\rm span}\{\wt w_2\}.$$
\end{lemma}
\begin{proof}
We first decompose  $R_{\lambda,\vartheta}=R_I+R_c,$ where  
\[
R_I[w]:=\begin{pmatrix}
\lambda^2( a^3w' )' \\[1ex]
 \lambda^2a^3(0) w'(0) 
\end{pmatrix}\qquad\text{and}\qquad 
R_c[w]:=\begin{pmatrix}
   - \vartheta a w- \lambda^2g\ov\rho' w  \\[1ex]
  -\big( \lambda^2 g\ov\rho(0) + \sigma\vartheta \big)w(0)
\end{pmatrix}.
\]
It is clear that $R_c$ is a compact operator. Furthermore, $R_I$ is an isomorphism.
Hence,  $R_{\lambda,\vartheta} $ is a Fredholm operator of index zero.

We now set $z_i:= a^3\wt w_i' - g\ov\rho \wt w_i$, $i\in\{1,\,2\}$, where $\wt w_1$ and $\wt w_2$ denote the unknowns in \eqref{SYS1} and \eqref{SYS2}, respectively.
Recalling the definition of $A$ in \eqref{AAA}, the equations  $\eqref{SYS1}_1$ and $\eqref{SYS2}_1$   can be recast as   a first order system of linear ODEs with continuous coefficients:
 \begin{equation}\label{SYSi}
\left\{
\begin{array}{lll}
\wt w_i':= A' \wt w_i+a^{-3}z_i,\\[1ex]
z_i' =(\lambda^{-2}\vartheta a -g \ov\rho A')\wt w_i-A' z_i 
\end{array}
\right.\qquad \text{for $i\in\{1,\,2\},$}
\end{equation}
and the classical theory, cf., e.g., \cite[Proposition 7.8]{Am90}, ensures that each of the problems $\eqref{SYS1}$ and $\eqref{SYS2}$ has
 a unique solution $\wt w_i\in W^2_r((p_0,0))$, $i\in\{1,\,2\}$.
Moreover, given $w_a,\, w_b\in W^2_{r}((p_0,0))$  solutions to $\eqref{BN}_1$, it follows that
\begin{equation}\label{Linde}
 a^3(w_a'w_b-w_aw_b')=C\qquad\text{in $[p_0,0]$}
\end{equation}
for some $C\in\R$.
Hence    $w_a$ and $w_b$ are colinear  if they also belong to $W^2_{r,0}((p_0,0))$.
This proves  in particular that   ${\rm dim} \ker R_{\lambda,\vartheta}\leq 1$.
It remains to establishing the last claim. Let  ${\rm dim} \ker R_{\lambda,\vartheta}= 1$ and choose $0\neq w\in  \ker R_{\lambda,\vartheta}$. Relation \eqref{Linde}
implies that $w$ and $\wt w_1$ are colinear. Invoking $\eqref{SLop}_2$ and $\eqref{SYS2}_2$, \eqref{Linde} shows that also $w $ and $\wt w_2$ are colinear.
Finally, if $\wt w_1$ and $\wt w_2$ are colinear, it is easy to see that they both belong to $ \ker R_{\lambda,\vartheta}.$
\end{proof}
  
We can now reformulate our task as the problem of determining $\lambda_*>0$ such that the Wronskian
\[
W(\,\cdot\,;\lambda,\vartheta):=\left|
\begin{array}{ccc}
\wt w_1&\wt w_2\\[1ex]
\wt w_1'&\wt w_2'
\end{array}
\right|
\] 
vanishes in $[p_0,0]$ only for $\vartheta=(2\pi)^2$.  
Recalling \eqref{Linde}, it follows that $W(\cdot;\lambda,\vartheta)$
vanishes in $[p_0,0]$ if and only if it vanishes at $p=0$. For this reason we consider the function $W(0;\cdot,\cdot):(0,\infty)\times\R\to\R$ defined by
\begin{align}
	\label{EQ:W0LT}
\begin{aligned}
W(0;\lambda,\vartheta)&= \wt w_1(0)\wt w_2'(0)-\wt w_1'(0)\wt w_2(0)=(\lambda^2 g\ov\rho(0) + \sigma\vartheta)\wt w_1(0)-\lambda^2a^3(0)\wt w_1'(0)\\[1ex]
&=  \sigma\vartheta \wt w_1(0)-\lambda^2z_1(0),
\end{aligned}
\end{align}
where $z_1=a^3\wt w_1'-g\ov\rho \wt w_1$ is the new variable   introduced in \eqref{SYSi}.
  By  \cite[Proposition 9.5]{Am90} it holds that 
  \begin{align}\label{Reg:W}
   W(0;\cdot,\cdot)\in{\rm C}^\infty((0,\infty)\times\R).
  \end{align}
 
 We next prove that for each $\lambda>0 $ there exists at least one solution   $\vartheta$ to $W(0;\lambda,\vartheta)=0$.
As a  first step we  show that 
\begin{align}\label{theta0}
W(0;\lambda,0)=-\lambda^2z_1(0)<0.
\end{align}
This is a direct consequence of the following more general statement.
 
 \begin{lemma}\label{L:the0}
 Assume that \eqref{C:the0} is satisfied. 
 If $\vartheta=0$ and $\lambda>0$, then $z_1>0$, $\wt w_1'>0$ in $[p_0,0]$ and $\wt w_1>0$ in $(p_0,0]$.
 \end{lemma}

 \begin{proof}
Arguing  as in the proof of Lemma \ref{L:w0} it follows that $z_1>0$ in $[p_0,0] $.
The remaining claims are direct consequences of  the latter property. Indeed, the relations  
\[
\wt w_1' - A' \wt w_1=\frac{z_1}{a^3} \quad \text{in $[p_0,0]$},\qquad \wt w_1(p_0)=0,
\]
imply
\begin{align}
\label{EQ:W1}
\wt w_1(p)=\int_{p_0}^p \frac{z_1(s)}{a^3(s)}e^{A(p)-A(s)}\, ds,  \qquad p\in[p_0,0].
\end{align}
As $z_1$ is positive, we conclude that  $\wt w_1>0$ in $(p_0,0]$ and $\wt w_1'>0$ in $[p_0,0]$.
\end{proof}

In view of  \eqref{Reg:W}, for the existence of a  solution $\vartheta$ to $W(0;\lambda,\vartheta)=0$
 it thus suffices to prove that $W(0;\lambda,\vartheta)\to\infty$ for $\vartheta\to\infty$.
We first show that if $\vartheta/\lambda^2$ is sufficiently large, then $\wt w_1'$ is a positive function.
This property is not obvious because of the fact that  $\ov\rho'$ has not only the opposed sign in $\eqref{BN}_1$  but it can also be unbounded.
However, when considering the equivalent formulation~\eqref{SYSi}, this feature follows quite naturally.

\begin{lemma}\label{L:inf1} Let $\lambda>0$ and  assume that
\begin{align}\label{Chopin}
\frac{\vartheta}{\lambda^2}\geq\frac{g^2\ov\rho^2(p_0)}{\underset{[p_0,0]}\min a^4}.
\end{align}
Then $z_1e^A$ and $\wt w_1$ are  increasing functions and it holds that
\begin{align}
	\label{ZPOS}
	z_1(p) \ge a^3(p_0)e^{-A(p)} > 0, \qquad \text{for all $p\in[p_0,0]$.}
\end{align}  
\end{lemma}
\begin{proof} Integrating \eqref{SYSi} (with $i=1$) yields
\begin{align*}
\wt w_1(p)=\int_{p_0}^p \frac{z_1(s)}{a^3(s)}e^{A(p)-A(s)}\, ds
\end{align*}
and 
  \begin{align}
  \label{EF}
  z_1(p)=a^3(p_0)e^{-A(p)}+\int_{p_0}^p\Big(\frac{\vartheta a(s)}{\lambda^2} -g\ov\rho(s)A'(s)\Big)e^{A(s)-A(p)}
  \int_{p_0}^s \frac{z_1(r)}{a^3(r)}e^{A(s)-A(r)}\, dr\, ds
  \end{align}    
  for all $p\in[p_0,0].$ 
 Hence, the assertions follow due to  \eqref{Chopin}.
\end{proof}

Combining \eqref{EQ:W0LT} and \eqref{EF} yields
\begin{equation}\label{FSSS}
\begin{aligned}
	W(0;\lambda,\vartheta)&=  \sigma\vartheta \int_{p_0}^0 \frac{z_1(s)}{a^3(s)}e^{A(0)-A(s)}\, ds\\[1ex]
	&\hspace{0,424cm}-\lambda^2\Big[a^3(p_0)e^{-A(0)}+\int_{p_0}^0\Big(\frac{\vartheta a(s)}{\lambda^2} -g\ov\rho(s)A'(s)\Big) 
	  \int_{p_0}^s \frac{z_1(r)}{a^3(r)}\frac{e^{2A(s)}}{e^{A(0)+A(r)} }\, dr\, ds\Big].
\end{aligned}
\end{equation}

\medskip

\begin{lemma}
	\label{LEM:LIMW}
 Let $\lambda >$0  be given and assume that condition \eqref{C:the0} is satisfied.
It then holds:
	\begin{align}
	\label{LIM:W}
		W(0;\lambda,\vartheta) \to \infty \quad\text{as}\quad \vartheta \to \infty.
	\end{align}
\end{lemma}

\begin{proof}
In the following, we assume that $\vartheta $ is large enough to ensure   \eqref{Chopin}.
Recall that, by Lemma \ref{L:inf1}, the function $z_1$ is positive and $z_1e^A$ is increasing on $[p_0, 0].$
Using \eqref{FSSS}, we obtain the estimate
\begin{equation}\label{FSSS1}
	\begin{aligned}
	W(0;\lambda,\vartheta)&\ge  C_1\vartheta\Big( \int_{p_0}^0  z_1(s) \, ds+ C_2 \int_{p_0}^0 s z_1(s) \, ds\Big)-\lambda^2 a^3(p_0)e^{-A(0)}
	\end{aligned}
\end{equation}
with positive constants $C_1$ and $C_2$ independent of $\vartheta$.
Therefore it suffices to show that
 \begin{equation}\label{FSSS2}
\liminf_{\vartheta\to\infty} \int_{p_0}^0  (1+ C_2   s) z_1(s) \, ds>0.
\end{equation}
Let $p_1\in [p_0,0]$ be an arbitrary number that is specified later. 
For $\vartheta$ sufficiently large   it holds that
\begin{align*}
\big(z_1(p)e^{A(p)}\big)'  
&=  \Big( \frac{\vartheta a(p)}{\lambda^2} -  g \ov\rho(p)A'(p) \Big)
	\int_{p_0}^p \frac{z_1(s)e^{A(s)}}{a^3(s)} e^{2(A(p)-A(s))} \, ds 
\ge C_3 \frac{\vartheta}{\lambda^2} \int_{p_1}^p z_1(s)e^{A(s)} \, ds
\end{align*}
for all $p\in[p_1,0]$ with a positive constant $C_3$  independent of $\vartheta$.
 We define
\begin{align*}
	Z(p):=\int_{p_1}^p z_1(s)e^{A(s)}\, ds		\quad \text{for all $p\in[p_1,0]$}
\end{align*}
and observe that
\begin{align*}
	Z'' \ge C_3\frac{\vartheta}{\lambda^2} Z \quad\text{on $[p_1,0]$}, \qquad Z(p_1) = 0, \qquad Z'(p_1) = z_1(p_1) e^{A(p_1)}.
\end{align*}
Now, let $U$ denote the solution of the initial value problem
\begin{align*}
	U'' = C_3\frac{\vartheta}{\lambda^2} U \quad\text{on $[p_1,0]$}, \qquad U(p_1) = 0, \qquad U'(p_1) = z_1(p_1) e^{A(p_1)}.
\end{align*}
Then, $U$ is explicitly given by
\begin{align*}
	U(p) = \frac{z_1(p_1) e^{A(p_1)}}{\mu} \sinh(\mu(p-p_1))\quad\text{with}\quad \mu := \sqrt{C_3\frac{\vartheta}{\lambda^2}}
\end{align*}
for all $p\in[p_1,0]$. 
It is straightforward to check that $Z\ge U$ on $[p_1,0]$. 
As $z_1$ is positive and $A$ is increasing, we can conclude that
\begin{align}
\label{EST:U}
	\int_{p_1}^0 z_1(s)\, ds 
	\ge e^{-A(0)} Z(0) 
	\ge e^{-A(0)} U(0) .
\end{align}
Now, we fix
\begin{align*}
	p_1:=\max\Big\{-\frac{1}{2C_2},\frac{p_0}{2} \Big\}. 
\end{align*}
Then, since $z_1$ is positive and $(z_1 e^A)$ is increasing, we use \eqref{EST:U} and the definition of $p_1$ to derive the estimate
\begin{align*}
	\int_{p_0}^0  (1+ C_2   s) z_1(s) \, ds&=\int_{p_0}^{p_1}  (1+ C_2   s) z_1(s) \, ds+\int_{p_1}^{0}  (1+ C_2   s) z_1(s) \, ds\\[1ex]
	&\geq \frac{1}{2}\int_{p_1}^{0}   z_1(s) \, ds -C_4\int_{p_0}^{p_1}   e^{A(s)} z_1(s) \, ds\\[1ex]
	&\geq \frac{1}{2}\Big(e^{-A(0)}U(0)  -C_4|p_0| e^{A(p_1)}   z_1(p_1)\Big) \\[1ex]
	&= \frac{z_1(p_1)}{2} 
		\Bigg(  \frac{e^{A(p_1)-A(0)}}{ \mu} \sinh(\mu |p_1|)
		- C_4|p_0| e^{A(p_1)}  \Bigg)
\end{align*}
where $C_4$ denotes a nonnegative constant that is independent of $\vartheta$. 
Hence, recalling \eqref{ZPOS}, we observe that the right-hand side tends to infinity as $\vartheta\to\infty$. In particular, the assertion \eqref{FSSS2} directly follows.
\end{proof} 

\medskip

The relations \eqref{Reg:W}, \eqref{theta0}, and  \eqref{LIM:W}  ensure that the equation $W(0;\lambda,\vartheta)=0$ has at least one solution $\vartheta>0$ for any fixed $\lambda>0$. 
The next result provides a remarkable identity, cf. \eqref{LEM:DW}, that will enable us later to identify the largest solution $\vartheta(\lambda)$
 to the above equation in a quite explicit way.

\begin{lemma}\label{L:38}
 	Assume that  \eqref{RES2} and  \eqref{RES3} hold and that $(\lambda,\vartheta)\in (0,\infty)^2$ satisfies $W(0;\lambda,\vartheta)=0$. 
 	Then, it holds $\wt w_1(0)W_\lambda(0;\lambda,\vartheta)< 0$ and
 	\begin{align}
 	\label{LEM:DW}
 		\cfrac{W_\vartheta(0;\lambda,\vartheta)}{W_\lambda(0;\lambda,\vartheta)} = -\frac{\lambda }{2 \vartheta }< 0.
 	\end{align}
\end{lemma}
\begin{proof}  
	Let $\wt w_1 = \wt w_1(\,\cdot\,;\lambda,\vartheta)$ denote the solution of \eqref{SYS1} corresponding  to $\lambda$ and $\vartheta$. 
We first consider the derivative $W_\lambda(0;\lambda,\vartheta).$ 
Using the algebra property of $W^1_r((p_0,0))$ we conclude that the partial derivative 
$\wt w_{1,\lambda} = \partial_\lambda \wt w_1(\,\cdot\,,\lambda,\vartheta)$ belongs to $W^2_r((p_0,0))$ and solves the problem
	\begin{align}
	\label{SYS:W1L}
	\left\{
	\begin{aligned}
		&\lambda^2 ( a^3\wt w_{1,\lambda}')' 
			- \lambda^2 g\ov\rho' \wt w_{1,\lambda}  
			- \vartheta a \wt w_{1,\lambda} 
			= -2\lambda   (a^3\wt w_{1}' )'
			+2\lambda g\ov\rho' \wt w_{1} \quad\text{in $L_r((p_0,0)),$}\\[1ex]
		&\wt w_{1,\lambda}(p_0)=0,\quad \wt w_{1,\lambda}'(p_0)=0.	
	\end{aligned}
	\right.
	\end{align}
	 We multiply $\eqref{SYS1}_1$ by $\wt w_{1,\lambda}$ and $\eqref{SYS:W1L}_1$ by $\wt w_1$ and subtract the resulting equations. This yields
	\begin{align*}
		\lambda^2 (a^3 \wt w_1')' \wt w_{1,\lambda} - \lambda^2 (a^3 \wt w_{1,\lambda}')' \wt w_{1} 
			= 2\lambda (a^3\wt w_1')' \wt w_1 - 2\lambda g\ov\rho' \wt w_1^2 \qquad \text{in $L_r((p_0,0))$}.
	\end{align*}
	Integrating with respect to $p$ from $p_0$ to $0$ and using integration by parts then gives
	\begin{align}
	\label{EQ:SUB}
		&\lambda^2 a^3(0) \wt w_1'(0) \wt w_{1,\lambda}(0) 
			- \lambda^2 a^3(0) \wt w_{1,\lambda}'(0) \wt w_{1}(0) \notag\\
		&\qquad = 2\lambda a^3(0)\wt w_1'(0) \wt w_1(0) 
			- 2\lambda \int_{p_0}^0\big[ a^3(p)\,  (\wt w_1')^2(p) + g\ov\rho'(p)\, \wt w_1^2(p) \big]\, dp.
	\end{align}
	Recall that $W(0;\lambda,\vartheta) = (\lambda^2g\ov\rho(0) + \sigma\vartheta)\wt w_1(0) - \lambda^2a^3(0)\wt w_1'(0)$ according to \eqref{EQ:W0LT}. 
	Hence, the derivative with respect to $\lambda$ is given by
	\begin{align*}
		W_\lambda (0;\lambda,\vartheta) = 2\lambda g\ov\rho(0) \wt w_1(0) 
			+ (\lambda^2 g\ov\rho(0) +\sigma\vartheta) \wt w_{1,\lambda}(0)
			- 2\lambda a^3(0) \wt w_1'(0) - \lambda^2 a^3(0) \wt w_{1,\lambda}'(0).
	\end{align*}
	We point out that, since $W(0;\lambda,\vartheta) = 0$, the functions $\wt w_1$ and $\wt w_2$ are linearly dependent and not identically zero.
	This implies that hence $\wt w_1(0)\neq0.$
	Multiplying the latter identity by $\wt w_1(0)$  and using \eqref{EQ:SUB}  and  the colinearity of $\wt w_1$ and $\wt w_2$  we then obtain
	\begin{align}
	\label{EQ:WL0}
		\wt w_1(0) W_\lambda (0;\lambda,\vartheta)
			&= 2\lambda \left( g\ov\rho(0) \wt w_1^2(0) 
				- \int_{p_0}^0 a^3(p)\, \big(\wt w_1'(p)\big)^2 \,dp  
				- \int_{p_0}^0 g\ov\rho'(p)\, \wt w_1^2(p) \, dp \right).
	\end{align} 
In view of
	\[
	- \int_{p_0}^0 g\ov\rho'(p)\, \wt w_1^2(p) \, dp\leq g(\ov\rho(p_0)-\ov\rho(0))\|\wt w_1\|_{{\rm C}([p_0,0])}^2
	\]
	we conclude that 
	\begin{align*}
		\wt w_1(0) W_\lambda (0;\lambda,\vartheta)
			&\leq 2\lambda \left( g\ov\rho(0) \big[\wt w_1^2(0)- \|\wt w_1\|_{{\rm C}([p_0,0])}^2\big]+g\ov\rho(p_0) \|\wt w_1\|_{{\rm C}([p_0,0])}^2
				- \int_{p_0}^0 a^3(p)\, \big(\wt w_1'(p)\big)^2 \,dp  
				 \right)\\[1ex]
			 &\leq 2\lambda \left( g\ov\rho(p_0) \|\wt w_1\|_{{\rm C}([p_0,0])}^2 
				- \int_{p_0}^0 a^3(p)\, \big(\wt w_1'(p)\big)^2 \,dp  \right).
	\end{align*}	
	Finally, choosing $p_1\in(p_0,0]$  such that $|\wt w_1(p_1)|=\|\wt w_1\|_{{\rm C}([p_0,0])}$, together with  \eqref{kl*} (recall that the positive solution $x_*$ to \eqref{exx} satisfies $x_*<2$) we get
	\begin{align*}
		 g\ov\rho(p_0)\|\wt w_1\|_{{\rm C}([p_0,0])}^2 &=  g\ov\rho(p_0)\Big(\int_{p_0}^{p_1}\wt w_1'(p)\, dp \Big)^2\leq g\ov\rho(p_0)\Big(\int_{p_0}^0\frac{1}{a^3(p)}\, dp\Big)
		 \Big(\int_{p_0}^0a^3(p)(\wt w_1'(p))^2\, dp\Big)\\[1ex]
		 &<\int_{p_0}^0a^3(p)(\wt w_1'(p))^2\, dp,
	\end{align*}
	and this proves that $\wt w_1(0) W_\lambda(0;\lambda,\vartheta)<0$. 
	
	We next consider  $W_\vartheta(0;\lambda,\vartheta)$. 
	The partial derivative $\wt w_{1,\vartheta} = \partial_\vartheta \wt w_1(\,\cdot\,,\lambda,\vartheta)$ belongs to $W^2_r((p_0,0))$ and solves the problem
	\begin{align}
	\label{SYS:W1T}
	\left\{
	\begin{aligned}
	&\lambda^2 ( a^3\wt w_{1,\vartheta}')' 
	- \lambda^2 g\ov\rho' \wt w_{1,\vartheta}  
	- \vartheta a \wt w_{1,\vartheta} 
	= a \wt w_1
	 \quad\text{in $L_r((p_0,0)),$}\\[1ex]
	&\wt w_{1,\vartheta}(p_0)=0,\quad \wt w_{1,\vartheta}'(p_0)=0.	
	\end{aligned}
	\right.
	\end{align}
	Multiplying $\eqref{SYS1}_1$ by $\wt w_{1,\vartheta}$ and $\eqref{SYS:W1T}_1$ by $\wt w_{1}$ and subtracting the resulting equations gives
	\begin{align*}
		\lambda^2 (a^3 \wt w_1')' \wt w_{1,\vartheta} - \lambda^2 (a^3 \wt w_{1,\vartheta}')' \wt w_{1} 
		= - a \wt w_1^2\qquad \text{in $L_r((p_0,0))$}.
	\end{align*}
	Integrating with respect to $p$ from $p_0$ to $0$ and using integration by parts, we infer that
	\begin{align}
	\label{EQ:SUB2}
		\lambda^2 a^3(0) \wt w_1'(0) \wt w_{1,\vartheta}(0) 
			- \lambda^2 a^3(0) \wt w_{1,\vartheta}'(0) \wt w_{1}(0)
			= - \int_{p_0}^0 a(p)\,  \wt w_1^2(p)\,  dp.
	\end{align}
	The partial derivative $W_\vartheta(0;\lambda,\vartheta)$ is given by
	\begin{align*}
		W_\vartheta(0;\lambda,\vartheta) = \sigma \wt w_1(0) + (\lambda^2g\ov\rho(0) + \sigma\vartheta) \wt w_{1,\vartheta}(0) - \lambda^2 a^3(0) \wt w_{1,\vartheta}'(0).
	\end{align*}
	Multiplying this equation by $\wt w_1(0)$ and using \eqref{EQ:SUB2} and the colinearity of  $\wt w_1$ and $\wt w_2$  we conclude that  
	\begin{align*}
		\wt w_1(0)W_\vartheta(0;\lambda,\vartheta) =  \sigma \wt w_1^2(0) - \int_{p_0}^0 a(p)\,  \wt w_1^2(p) \, dp .
	\end{align*}
	Next, we multiply $\eqref{SYS1}_1$ by $\wt w_1 $, integrate from $p_0$ to $0$, and use once more the fact that $\wt w_1$ and $\wt w_2$ are linearly dependent to obtain 
	\begin{align*}
		\sigma \wt w_1^2(0) - \int_{p_0}^0 a(p)\,  \wt w_1^2(p) \, dp 
			= -\frac{\lambda^2}{\vartheta} \left( g \ov\rho(0) \wt w_1^2(0)
			-  \int_{p_0}^0 a^3(p) (\wt w_1'(p))^2 \,dp
			- \int_{p_0}^0 g\ov\rho'(p) \wt w_1^2(p) \,dp.
			 \right)
	\end{align*}
	Hence, \eqref{EQ:WL0} implies   that
	\begin{align}\label{WOW}
		W_\vartheta(0;\lambda,\vartheta) = -\frac{\lambda }{2 \vartheta } W_\lambda(0;\lambda,\vartheta),
	\end{align}
	which completes the proof.
\end{proof}

Given $\lambda>0$, let $\vartheta(\lambda)$ denote the largest solution to   $W(0;\lambda,\vartheta)=0. $
In the following lemma  we identify, using Lemma \ref{L:38}, the mapping  
 \begin{align}\label{varla}
[\lambda\mapsto\vartheta(\lambda)]:(0,\infty)\to(0,\infty) 
 \end{align}
up to a positive multiplicative constant.
 
 \begin{lemma}\label{L:39} 
 There exists a constant $C_D>0$ such that   $\vartheta(\lambda)=C_D\lambda^2, \, \lambda>0.$
 \end{lemma}
 \begin{proof} 
Lemma \ref{LEM:LIMW} and Lemma \ref{L:38} ensure  that $W_\vartheta(0;\lambda,\vartheta(\lambda))>0$ for all $\lambda>0$. The implicit function theorem applied at
  $(\lambda_0,\vartheta(\lambda_0))$, with $\lambda_0>0$, implies that in a small neighborhood of $(\lambda_0,\vartheta(\lambda_0))$ 
  the solution set to $W(0;\lambda,\vartheta)=0$ coincides with the graph of a smooth curve $\wt\vartheta:(\lambda_0-\e,\lambda_0+\e)\to\R$.
 Differentiating the relation $W(0;\lambda,\wt\vartheta(\lambda))=0$ it follows in virtue of \eqref{WOW} that
 \[
\wt\vartheta'(\lambda)=-\frac{W_\lambda(0;\lambda,\wt\vartheta(\lambda))}{W_\vartheta(0;\lambda,\wt\vartheta(\lambda))} =\frac{2 \wt\vartheta(\lambda) }{\lambda}.
 \] 
Hence there exists a constant $C_D>0$ such that $\wt\vartheta(\lambda)=C_D\lambda^2$ for all $\lambda\in (\lambda_0-\e,\lambda_0+\e).$
The desired claim follows now at once.
 \end{proof}

 \begin{rem}\label{R:CD} Our analysis shows, under the assumptions \eqref{RES2} and \eqref{RES3}, 
 that the constant $C_D$ found in Lemma \ref{L:39} is  the largest constant  such that the solutions to 
 \[
 ( a^3w' )' - g\ov\rho' w - C_D a w=0\qquad\text{in $L_r((p_0,0))$}
 \]
 determined by the initial data
 \[
(w,w')(p_0)=(0,1)\qquad\text{or} \qquad  (w,w')(0)=(a^3(0),g\ov\rho(0) + \sigma C_D),\, \text{respectively,}
 \]
 are linearly dependent.
 The constant  $C_D$ depends only on Earth's gravity $g$, the mass flux $p_0$, the water depth $d$, the surface tension coefficient $\sigma$, the density function $\ov\rho,$ and on Bernoulli's function~$\beta$.
 In the homogeneous case $\ov\rho=\rho\in\R$ we obtain, under the assumption that the flow is irrotational (that is for $\beta=0$),
 that $C_D$ is the largest positive solution to 
\begin{equation}\label{DR}
\tanh(d\sqrt{C_D} )=\frac{p_0^2\sqrt{C_D}}{d^2(g\rho+\sigma C_D)}.
\end{equation}

We cannot exclude the possibility that there exist finitely many (since $W(0;\lambda,\cdot):\R\to\R$ is real-analytic there cannot exist infinitely many) 
positive constants $$ C_{D,N}<C_{D,{N-1}}<\ldots C_{D,1}<C_D,\quad N\geq 1,$$ 
for which the two solutions defined above   are linearly dependent. 
 Each of these constants defines a new function $\vartheta_i:(0,\infty)\to\R$ with
 \[ 
 \vartheta_i(\lambda)=C_{D,i}\lambda^2, \quad 1\leq i\leq N,
 \]
   which  satisfies
 $W(0;\lambda,\vartheta_i(\lambda))=0 $ for all $\lambda>0$.
 This complicates the bifurcation analysis for \eqref{BP1} a lot as in this situation the dimension of $\ker\p_h\kF(\lambda,0)$ may   be larger than $1$ for certain $\lambda$.
 However, this behavior is expected since, even in the case of constant density,  phenomena like bifurcation  from double eigenvalues or secondary bifurcation 
 may occur when allowing for surface tension effects, cf. \cite{MJ89, Wa14a, Wal14b}. 
 
We now shortly discuss  the dispersion relation \eqref{DR} in the homogeneous irrotational case (that is for $\ov\rho=\rho\in\R$  and  $\beta=0$). 
Setting $\sqrt{C_D}=x$, the problem reduces to finding the positive zeros of the function ${f:\R\to\R}$ with
\begin{equation}\label{DRD}
f(x):=(g\rho+\sigma x^2)\frac{\tanh(dx)}{x}-\frac{p_0^2}{d^2}.
\end{equation}
Since $f(0)=g\rho d-p_0^2/d^2<0$ (this inequality is a direct consequence of \eqref{RES3}) and $f(x)\to\infty$ as $x\to \infty$, the equation \eqref{DRD} has at least a positive solution.
 Note that $f$ is even. Furthermore, direct
computations show that 
\[
  f''(0)=2d\Big(\sigma-\frac{g\rho d^2}{3}\Big) \qquad\text{and}\qquad f^{(4)}(0)=-8d^3\Big(\sigma-\frac{2g\rho d^2}{5}\Big).
\] 
Moreover, it can be shown that if $f'(x)=0$ for some $x>0$, then $f''(x)>0$ (see e.g., {\cite[Lemma~3]{CM13a}}).
Consequently, regardless of the sign of 
$$\sigma-\frac{g\rho d^2}{3}$$
the equation $f(x)=0$ has a unique solution. 
Indeed, if $\sigma > g\rho d^2/3$, then $f$ is strictly increasing on $(0,\infty)$. If $\sigma=g\rho d^2/3$ then $f$ is again strictly 
increasing on $(0,\infty)$ since then $f^{(4)}(0)>0$. Finally,
 if $\sigma< g\rho d^2/3$, then $f$ has a unique global minimizer let's say at $x_0$ and $f$ is strictly decreasing on $(0,x_0)$ and strictly increasing on $(x_0,\infty)$. 

 \end{rem}

 We proceed with the following result.
 \begin{prop}\label{Prop:1} Let 
 \begin{align}\label{lambda*}
  \lambda_*:=\frac{2\pi}{\sqrt{C_D}}.
 \end{align}
 Then, $\p_h\kF(\lambda_*,0)$ is a Fredholm operator of index zero and with a one-dimensional kernel spanned by
 \begin{align}\label{w*}
 w_*(q,p):=\wt w_1(p;\lambda_*, (2\pi)^2)\cos(2\pi q), \qquad (p,q)\in\0, 
 \end{align}
 where $\wt w_1(\cdot;\lambda_*, (2\pi)^2)$ denotes the solution to \eqref{SYS1} corresponding to $(\lambda,\vartheta)=(\lambda_*,(2\pi)^2).$
 \end{prop}
 \begin{proof}
 The proof follows from the results established in Lemma \ref{L:31}, Lemma \ref{L:w0}, Lemma~\ref{LEM:KER},   Lemma~\ref{L:39}, and Remark \ref{R:CD}.
 \end{proof}

\paragraph{\bf The transversality condition.} In order to apply  \cite[Theorem~1.7]{CR71} to the bifurcation problem \eqref{BP1}, we  still have to check  the transversality condition
 \begin{align}\label{TC}
 \p_{\lambda h}\kF(\lambda_*,0)[w_*]\not\in \im \p_{h}\kF(\lambda_*,0),
 \end{align}
  with $\lambda_*$ and $w_*$ introduced in \eqref{lambda*} and \eqref{w*}, respectively.
 To this end we first characterize the set $\im \p_{h}\kF(\lambda_*,0).$ 
 
 \begin{lemma}\label{L:40}
 A pair $(f,\phi)\in \Y$ with $f=\varphi+\p_q\psi_1+\p_p\psi_2$ belongs to $\im \p_{h}\kF(\lambda_*,0)$
 if and only if
 \begin{align}\label{IF0}
 \int_\0\big(\psi_1w_{*,q}+\psi_2w_{*,p}-\varphi w_*\big)\, d(q,p)- \int_0^1\tr_0(\psi_2w_{*})\,dq-\sigma(1+(2\pi)^2)\int_0^1\phi\tr_0w_{*}\,dq=0.
 \end{align}
  \end{lemma}
  \begin{proof}
  Let $(f,\phi)\in  \im \p_{h}\kF(\lambda_*,0)$ and let $w\in \X$ satisfy $(L,T)[w]=(f,\phi).$
  Testing the equation $L[w]=f$ with $\xi_m w_*\in H^1_0(\0)$, where $w_*$ is defined  in \eqref{w*} and with $\xi_m$ defined by
  \[
\xi_m(p):=\left\{
\begin{array}{clll}
m(p-p_0),& p_0\leq p\leq p_0+1/m,\\[1ex]
1,&   p_0+1/m\leq p\leq -1/m,\\[1ex]
-mp,&  -1/m\leq p\leq0,
\end{array}
\right.  \qquad p\in[p_0,0], \, \frac{2}{|p_0|}\leq m\in\N,
  \]
 we obtain, after passing to the limit $m\to\infty$, the following identity 
\begin{equation}\label{IF1}
 \begin{aligned}
 &\int_\0\Big(\frac{w_qw_{*,q}}{H'}+\lambda^2 \frac{w_pw_{*,p}}{H'^3}+\lambda^2g\ov\rho' ww_*\Big)\, d(q,p)
 -\lambda^2\int_0^1\tr_0 \frac{w_pw_{*}}{H'^3} \,dq\\[1ex]
 &\hspace{0.5cm}=\int_\0\big(\psi_1w_{*,q}+\psi_2w_{*,p}-\varphi w_*\big)\, d(q,p)- \int_0^1\tr_0(\psi_2w_{*})\,dq.
 \end{aligned}
 \end{equation}
Moreover, multiplying the relation $T[w]=\phi$ by $\tr_0w_*$ and integrating over $[0,1],$ it follows that
\begin{equation}\label{IF2}
   \int_0^1\tr_0[(\sigma(2\pi)^2+\lambda^2g\ov\rho)ww_*]\, dq-\lambda^2\int_0^1\tr_0 \frac{w_pw_{*}}{H'^3}\,dq =\sigma(1+(2\pi)^2)\int_0^1\phi\tr_0w_{*}\,dq.
 \end{equation}
Combining the  relations \eqref{IF1} and \eqref{IF2} yields 
 \begin{equation}\label{IF3} 
 \begin{aligned}
 &\int_\0\Big(\frac{w_qw_{*,q}}{H'}+\lambda^2 \frac{w_pw_{*,p}}{H'^3}+\lambda^2g\ov\rho' ww_*\Big)\, d(q,p)-\int_0^1\tr_0[(\sigma(2\pi)^2+\lambda^2g\ov\rho)ww_*]\, dq\\[1ex]
 &\hspace{0.5cm}=\int_\0\big(\psi_1w_{*,q}+\psi_2w_{*,p}-\varphi w_*\big)\, d(q,p)- \int_0^1\tr_0(\psi_2w_{*})\,dq-\sigma(1+(2\pi)^2)\int_0^1\phi\tr_0w_{*}\,dq.
 \end{aligned}
 \end{equation}
 Finally, multiplying  $ L [w_*]=0$ by $w$ and integrating by parts, we find in virtue of 
\[
\tr_0w_{*,p}=\tr_0\Big[(\sigma(2\pi)^2+\lambda^2g\ov\rho)\frac{H'^3}{\lambda^2}w_*\Big],
\] 
 cf. \eqref{w*}, that the left-hand side of \eqref{IF3} is zero and \eqref{IF0} follows.
 Observing that \eqref{IF0} defines a closed subspace of $\Y$ of codimension $1$ which   contains $\im \p_{h}\kF(\lambda_*,0)$, Proposition \ref{Prop:1} leads us to the desired conclusion.
  \end{proof}
  
  We are now in a position to prove that the transversality condition \eqref{TC} holds.
  
  \begin{prop}\label{Prop:2}
  Let $\lambda_*$ and $w_*$ be as defined in Proposition \ref{Prop:1}.
   It then holds 
   \begin{align*} 
 \p_{\lambda h}\kF(\lambda_*,0)[w_*]\not\in \im \p_{h}\kF(\lambda_*,0).
 \end{align*}
  \end{prop}
  \begin{proof}
  We compute that 
  \begin{align*}
   \p_{\lambda h}\kF(\lambda_*,0)[w_*]=2\lambda_*(\varphi+ \p_p\psi ,\phi),
  \end{align*}
  where
  \[
  \varphi=g\ov\rho w_{*,p},\quad \psi =\frac{w_{*,p}}{H'^3}-g\ov\rho w_{*},
\quad \phi=-\frac{1}{\sigma(1+(2\pi)^2)}\tr_0\psi.
  \]
  Hence, according to Lemma \ref{L:40}, $\p_{\lambda h}\kF(\lambda_*,0)[w_*]\in \im \p_{h}\kF(\lambda_*,0)$ if and only if
  \begin{align*}
  \int_\0\big( \psi w_{*,p}-\varphi w_*\big)\, d(q,p)- \int_0^1\tr_0(\psi w_{*})\,dq-\sigma(1+(2\pi)^2)\int_0^1\phi\tr_0w_{*}\,dq=0.
  \end{align*}
The   left-hand side of this equation can be expressed as
  \begin{align*}
 & \int_\0\Big( \frac{w_{*,p}^2}{H'^3}-g\ov\rho (w_*^2)_p\Big) \,d(q,p)=\frac{1}{2}\int_{p_0}^0\Big( \frac{(\wt w_1')^2}{H'^3}-g\ov\rho (\wt w_1^2)'\Big) \,dp  
 =-\frac{\wt w_1(0)}{4\lambda_*}W_\lambda(0;\lambda_*,(2\pi)^2),
  \end{align*}
  which  is positive according to Lemma \ref{L:38}. 
  Thus, the assertion follows.
  \end{proof}

 \paragraph{\bf Improved regularity and the proof of Theorem \ref{MT1}.} With this preparation completed we may now apply the bifurcation result \cite[Theorem~1.7]{CR71} to~\eqref{BP1}.
  This provides us  with  a local branch of  weak solutions  to~\eqref{eq:Q'} and~\eqref{C4}  which contains, 
  apart from the laminar flow defined by $(\lambda_*,0)$ only nonlaminar solutions which belong to $ \X$. 
  We next prove that any weak solution $h\in {\rm C}^{1+\alpha}(\ov\0)$  to~\eqref{eq:Q'} and~\eqref{C4},  corresponding to some $\lambda>0$,  is in fact 
  a strong solution as defined in Theorem \eqref{T:Ep} (iii).

 To this end we need the following regularity result.
 \begin{thm}\label{T:RR}
 Let $\lambda>0$, $r\in(1,\infty)$, $\alpha=(r-1)/r$, $\beta\in L_r((p_0,0))$ and $\ov\rho\in W^1_{r}((p_0,0)).$ 
Given  a weak solution  $h\in {\rm C}^{1+\alpha}(\ov\0)$  to~\eqref{eq:Q'} and~\eqref{C4}, it holds that $\p_q^m h\in {\rm C}^{1+\alpha}(\ov\0)$ for all $m\in\N$ and there exists a  constant 
 $L>0$ such that 
 \begin{align}\label{L:AST}
 \|\p_q^m h\|_{{\rm C}^{1+\alpha}(\ov\0)}\leq L^{m-2}(m-3)!\qquad\text{for all $m\geq 3.$}
 \end{align}
 \end{thm}
 
 The proof of Theorem \ref{T:RR} is quite technical, but very similar to that of \cite[Theorem 5.1]{MM14} and is therefore omitted.

  \begin{rem}\label{R:3}
     The estimate \eqref{L:AST} implies in particular that all the streamlines, including the wave surface, of the corresponding strong solution to~\eqref{eq:Q'} and~\eqref{C4}, 
     see Proposition~\ref{P:51} and Theorem~\ref{T:Ep},
  are real-analytic graphs.
   Similar results  for classical solutions to~\eqref{eq:Q'} and~\eqref{C4} have been obtained in \cite{HM,W13} 
   under the more restrictive assumption  that $\ov\rho'$, $\beta\in{\rm C}^\alpha([p_0,0])$. 
   We point out that the study of the a priori regularity  of homogeneous but rotational waves has been initiated in \cite{EC11}, see also \cite{EM14}.
  \end{rem}
  
  The next lemma plays an important role in the proof of Theorem~\ref{MT1}~(iii).
  \begin{lemma}\label{L:LLC}
  Given two solutions  $(\lambda_i,h_i)\in(0,\infty)\times\X$, $i=1,\, 2$,  to \eqref{BP1}, there exists a constant
   $C=C(\lambda_1,\lambda_2,\|h_1\|_{{\rm C^{1+\alpha}(\ov\0)}},\|h_2\|_{{\rm C^{1+\alpha}(\ov\0)}},\|\p_qh_{2}\|_{{\rm C^{1+\alpha}(\ov\0)}})$ such that  
  \begin{align}\label{DEP}
  \|\p_qh_1-\p_qh_2\|_{{\rm C^{1+\alpha}(\ov\0)}}\leq C(|\lambda_1-\lambda_2|+ \|h_1-h_2\|_{{\rm C^{1+\alpha}(\ov\0)}}).
  \end{align}
  \end{lemma}
  \begin{proof}
  Since $w_i:=\p_qh_i\in {\rm C^{1+\alpha}(\ov\0)}$, cf. Theorem \ref{T:RR}, differentiation of the
   relations in Definition~\ref{D:31} shows that $w_i$ is a weak solution to the uniformly elliptic boundary value problem
\begin{equation*} 
\left\{
\begin{array}{rrllll}
\big(a_{i,11}w_{i,q} \big)_q+\big(a_{i,21}w_{i,p}\big)_q+\big(a_{i,12} w_{i,q}\big)_p+
\big( a_{i,22}w_{i,p}+b_i w_i\big)_p+c_i w_{i,p}&=&0 & \text{in $ \0,$}\\[1ex]
 w_i-(1-\p_q^2)^{-1}\tr_0(A_iw_i+B_iw_{i,q}+C_iw_{i,p})&=&0 & \text{on $p=0$,}\\[1ex]
w_i&=&0 & \text{on $ p=p_0$},
\end{array}
\right.
\end{equation*}
where $a_{i,k\ell},\, b_i, \, c_i,\, A_i,\, B_i,\, C_i\in {\rm C^{\alpha}(\ov\0)}$ are given by
\begin{align*}
a_{i,11}&=\cfrac{1}{h_{i,p}},\qquad a_{i,12}=a_{i,21}=-\cfrac{h_{i,q}}{h_{i,p}^2},\qquad a_{i,22} =\cfrac{\lambda_i^2+h_{i,q}^2}{h_{i,p}^3},
 \quad b_i=-\lambda_i^2g\ov\rho, \qquad c_i=\lambda_i^2g\ov\rho,\\[1ex]
A_i&=1-g \ov\h\frac{ (\lambda_i^2+h_{i,q}^2)^{3/2}}{\sigma\lambda_i },\qquad C_i=\frac{ (\lambda_i^2+h_{i,q}^2)^{5/2} }{\sigma\lambda_i ^3h_{i,p}^3},  \\[1ex]
B_i&=-\frac{3(\lambda_i^2+h_{i,q}^2)^{1/2}h_{i,q}}{2\sigma\lambda_i^3}\left(\frac{\lambda_i^2+h_{i,q}^2}{h_{i,p}^2}
+2\lambda_i^2g \ov\h(h_i-d)-\int_{0}^1\frac{\lambda_i^2+h_{i,q}^2}{h_{i,p}^2}\, dq\right)
-\frac{ (\lambda_i^2+h_{i,q}^2)^{3/2}h_{i,q}}{\sigma\lambda_i ^3h_{i,p}^2}.
\end{align*}
Hence, for the difference $W:=w_1-w_2,$ we find that
   \begin{equation*} 
\left\{
\begin{array}{rrllll}
\big(a_{1,11}W_{q} \big)_q+\big(a_{1,21}W_{p}\big)_q+\big(a_{1,12} W_{q}\big)_p+
\big( a_{1,22}W_{p}+b_i W\big)_p+c_1 W_{p} &=& f &\text{in $ \0,$}\\[1ex]
 W&=& \varphi & \text{on $p=0$},\\[1ex]
W&=&0 & \text{on $ p=p_0$},
\end{array}
\right.
\end{equation*}
where
\begin{align*}
f&=\big((a_{2,11}-a_{1,11})w_{2,q} \big)_q+\big((a_{2,21}-a_{1,21})w_{2,p}\big)_q+\big((a_{2,12}-a_{1,12}) w_{2,q}\big)_p\\[1ex]
&\hspace{0.424cm}+\big( (a_{2,22}-a_{1,22})w_{2,p}+(b_2-b_1) w_2\big)_p+(c_2-c_1) w_{2,p},\\[1ex]
\varphi&=(1-\p_q^2)^{-1}\tr_0(A_1W+B_1W_{q}+C_1W_{p}+w_2(A_1-A_2)+w_{2,q}(B_1-B_2)+w_{2,p}(C_1-C_2)).
\end{align*}
Using \cite[Theorem 8.33]{GT01}, we estimate
\begin{align*}
\|W\|_{{\rm C^{1+\alpha}(\ov\0)}}&\leq C\big( \|W\|_{{\rm C(\ov\0)}}+ \|\varphi\|_{{\rm C^{1+\alpha}(\R)}}+ \|(c_2-c_1) w_{2,p}\|_{{\rm C(\ov\0)}}
+ \|(a_{2,11}-a_{1,11})w_{2,q}\|_{{\rm C^\alpha(\ov\0)}}\\[1ex]
&\hspace{1cm}+ \|(a_{2,21}-a_{1,21})w_{2,p}\|_{{\rm C^\alpha(\ov\0)}}+\|(a_{2,12}-a_{1,12}) w_{2,q}\|_{{\rm C^\alpha(\ov\0)}}\big)\\[1ex]
&\hspace{1cm}+\|(a_{2,22}-a_{1,22})w_{2,p}+(b_2-b_1) w_2\|_{{\rm C^\alpha(\ov\0)}}\big)\\[1ex]
&\leq C\big( \|W\|_{{\rm C^{1+\alpha/2}(\ov\0)}}+ \|\varphi\|_{{\rm C^{1+\alpha}(\R)}}+|\lambda_1-\lambda_2|+\|h_2-h_1\|_{{\rm C^{1+\alpha}(\ov\0)}}\big),
\end{align*}
where, in virtue of $(1-\p_q^2)^{-1}\in\kL({{\rm C^{\alpha/2}(\R)}}, {{\rm C^{2+\alpha/2}(\R)}})$, it holds that
\begin{align*}
\|\varphi\|_{{\rm C^{1+\alpha}(\R)}}&\leq C\|\varphi\|_{{\rm C^{2+\alpha/2}(\R)}}\\[1ex]
&\leq C\| A_1W+B_1W_{q}+C_1W_{p}+w_2(A_1-A_2)+w_{2,q}(B_1-B_2)+w_{2,p}(C_1-C_2) \|_{{\rm C^{\alpha/2}(\ov\0)}}\\[1ex]
&\leq C\big( \|W\|_{{\rm C^{1+\alpha/2}(\ov\0)}}+  |\lambda_1-\lambda_2|+\|h_2-h_1\|_{{\rm C^{1+\alpha}(\ov\0)}}\big).
\end{align*}
Hence
\begin{align*}
\|W\|_{{\rm C^{1+\alpha}(\ov\0)}}& \leq C\big( \|W\|_{{\rm C^{1+\alpha/2}(\ov\0)}}+|\lambda_1-\lambda_2|+\|h_2-h_1\|_{{\rm C^{1+\alpha}(\ov\0)}}\big),
\end{align*}
and the desired claim follows by the interpolation result $ {\rm C^{1+\alpha/2}(\ov\0)}=({\rm C^{\alpha}(\ov\0)},{\rm C^{1+\alpha}(\ov\0)})_{1-\alpha/2,\infty}.$
  \end{proof}

  We next prove that weak solutions to~\eqref{eq:Q'} and \eqref{C4} are in fact strong solutions (as defined Theorem~\ref{T:Ep}~(iii)).
  
  \begin{prop}\label{P:51} Let $\lambda>0$ and let $h\in {\rm C}^{1+\alpha}(\ov\0)$ denote a weak solution to~\eqref{eq:Q'} and \eqref{C4}.
  Then $h\in W^2_r(\0)$ is a strong solution to~\eqref{eq:Q'} and~\eqref{C4}.
  \end{prop}
  \begin{proof}
  In virtue of Theorem \ref{T:RR} we have $h_q\in {\rm C}^{1+\alpha}(\ov\0),$ hence $\p_qh_p=\p_ph_q\in {\rm C}^{\alpha}(\ov\0).$ 
  This in turn implies that 
  $\p_qh_p$ is the classical derivative of $h_p$ with respect to $q$.  
  Recalling that 
  \begin{align*}
 \Big(\frac{h_q}{h_p}\Big)_q - \Big(\frac{\lambda^2 + h_q^2}{2h_p^2} + \lambda^2 B + \lambda^2 g \ov\rho (h-d) \Big)_p +\lambda^2 g \ov\rho h_p=0\qquad\text{in $\mathcal{D}'(\0)$},
\end{align*}
it follows that
 \begin{align*}
  \Big(\frac{\lambda^2 + h_q^2}{2h_p^2} + \lambda^2 B + \lambda^2 g \ov\rho (h-d) \Big)_q\in {\rm C}^{\alpha}(\ov\0),\quad 
    \Big(\frac{\lambda^2 + h_q^2}{2h_p^2} + \lambda^2 B + \lambda^2 g \ov\rho (h-d) \Big)_p\in {\rm C}^{\alpha}(\ov\0),
\end{align*}
and  \eqref{PR3} in turn yields
 \begin{align*}
  \frac{\lambda^2 + h_q^2}{2h_p^2} + \lambda^2 B + \lambda^2 g \ov\rho (h-d) \in {\rm C}^{1+\alpha}(\ov\0).
\end{align*}
Consequently, 
\[
 \frac{\lambda^2 + h_q^2}{2h_p^2}\in W^1_r(\0)\cap {\rm C}^{\alpha}(\ov\0)
\]
and the repeated use of \cite[Lemma 7.5]{GT01} finally yields $h_p\in W^1_r(\0)$. Since  $h_q\in {\rm C}^{1+\alpha}(\ov\0)$, it follows  that $h\in W^2_r(\0)$ is
 a strong solution to~\eqref{eq:Q'} and~\eqref{C4} (as ${\rm tr}_0 \, h$ is real-analytic, the condition ${\rm tr}_0\, h\in W^2_r(\R)$ is obvious).
  \end{proof}

  We complete this section with the proof of  Theorem \ref{MT1}.
  
  \begin{proof}[Proof of Theorem \ref{MT1}.]
 Gathering \eqref{TS}, \eqref{REG}, Proposition~\ref{Prop:1}, and Proposition \ref{Prop:2} we find that all  assumptions of  the theorem on bifurcation from simple eigenvalues
 by Crandall and Rabinowitz, cf.  \cite[Theorem~1.7]{CR71}, are satisfied in the context of  the bifurcation problem \eqref{BP1}.
  This abstract result yields the existence  of a local smooth curve
   \begin{align}\label{BCu1}
      [s\mapsto (\lambda(s),h(s))]:(-\e,\e)\to (0,\infty)\times \X,
   \end{align}
where  $\e>0$  is in general a small number,  such that $\mathcal{F}(\lambda(s),h(s))=0$ for all $|s|<\e$.
 Moreover, $\lambda(0)=\lambda_*$ and 
  \begin{align}\label{AAS0}
  h(s)=s(w_*+\chi(s)) 
  \end{align}
  with $\chi\in{\rm  C}^\infty((-\e,\e),\X)$ satisfying $\chi(0)=0$.
  Besides, there exists a ball in $(0,\infty)\times \X$, with  center  $(\lambda_*,0)$,  which does not contain other solutions but those mentioned above or trivial solutions 
  $(\lambda,0)$.
 Letting $H$ denote the laminar flow solution  fixed at the beginning of the section, it follows from Proposition \ref{P:51} that
$H+h(s)$ is a strong solution   to~\eqref{eq:Q'} and~\eqref{C4} having minimal wavelength $1$ for $s\neq0$.
Each pair $(\lambda(s),H+h(s))$ corresponds to a solution  $(u(s)-c,v(s),P(s),\rho(s),\eta(s))$  to~\eqref{eq:1}--\eqref{C1} which lies on the curve $\mathcal{C}$ in Theorem \ref{MT1}
and which has minimal period $\lambda(s)$ (provided that $s\neq0$), cf. \eqref{Scal} and Theorem \ref{T:Ep}.
 This proves the claims (i), (ii), and (iv) of Theorem~\ref{MT1}.
 
It remains to show  that $\eta(s)$ has,  for $s\neq0$, precisely one maximum and one
minimum in $[0,\lambda(s))$  and that $\eta(s)$ it is strictly monotone between
the points where the global extrema are attained.
 To this end we claim that
\begin{align}\label{AAS}
   \p_q\chi(s)\underset{s\to0}\to0 \quad \text{in ${\rm C}^{1+\alpha/2}(\ov\0)$.}
  \end{align}
Indeed, 
 Lemma \ref{L:LLC}  and \eqref{w*} imply that there exists a constant $C>0$ such that 
\begin{align*}
  \| \p_q\chi(s)\|_{{\rm C}^{1+\alpha}(\ov\0)}\leq \Big\|\frac{\p_qh(s)-\p_qh(0)}{s}\Big\|_{{\rm C}^{1+\alpha}(\ov\0)}+\|\p_qw_*\|_{{\rm C}^{1+\alpha}(\ov\0)}\leq C  
  \end{align*}
  for all $|s|<\e/2$.
  Additionally,  using the  differentiability  of $h$ at $s=0$,  we have
  \begin{align*}
  \| \p_q\chi(s)\|_{{\rm C}^{\alpha}(\ov\0)}= \Big\|\frac{\p_qh(s)-\p_qh(0)-s\p_qh'(0)}{s}\Big\|_{{\rm C}^{\alpha}(\ov\0)}
  \leq\Big\|\frac{h(s)-h(0)-sh'(0)}{s}\Big\|_{{\rm C}^{1+\alpha}(\ov\0)}\underset{s\to0}\to0.  
  \end{align*}
  These relations together with the interpolation result $ {\rm C^{1+\alpha/2}(\ov\0)}=({\rm C^{\alpha}(\ov\0)},{\rm C^{1+\alpha}(\ov\0)})_{1-\alpha/2,\infty}$ immediately yield \eqref{AAS}. 
  Since $h(s,\cdot,0)=s(w_*(\cdot,0)+\chi(s,\cdot,0))$ and $\chi(s,\cdot,0)\to0$ in ${\rm C}^{2}(\R)$ for $s\to0$, cf. \eqref{AAS0}-\eqref{AAS},
  standard arguments show that  the monotonicity properties of
  $w_*$ are inherited  by  $h(s,\cdot,0)$ provided that $\e$ is sufficiently small (see, e.g., \cite{W06b}). 
  Recalling that the wave profile is parametrized by the function $\eta(s)=h(s,\cdot,0)-d,$ we have established (iii) and the proof is complete. 
   \end{proof}\bigskip

\section*{Acknowledgement}
Patrik Knopf and Bogdan-Vasile Matioc were partially supported by the RTG 2339 ''Interfaces, Complex Structures, and Singular Limits''
of the German Science Foundation (DFG). The support is gratefully acknowledged.

\bibliographystyle{siam}
\bibliography{EKLM}

\begin{thebibliography}{10}

\bibitem{Alt16}
{\sc H.~W. Alt}, {\em {Linear Functional Analysis}}, Universitext,
  Springer-Verlag London, Ltd., London, 2016.
\newblock An application-oriented introduction, Translated from the German
  edition by Robert N\"{u}rnberg.

\bibitem{Am90}
{\sc H.~Amann}, {\em {Ordinary Differential Equations}}, vol.~13 of {de Gruyter
  Studies in Mathematics}, Walter de Gruyter \& Co., Berlin, 1990.
\newblock An introduction to nonlinear analysis, Translated from the German by
  Gerhard Metzen.

\bibitem{ASW16}
{\sc D.~M. Ambrose, W.~A. Strauss, and J.~D. Wright}, {\em Global bifurcation
  theory for periodic traveling interfacial gravity-capillary waves}, Ann.
  Inst. H. Poincar\'{e} Anal. Non Lin\'{e}aire, 33 (2016), pp.~1081--1101.

\bibitem{AT89}
{\sc C.~J. Amick and R.~E.~L. Turner}, {\em Small internal waves in two-fluid
  systems}, Archive for Rational Mechanics and Analysis, 108 (1989),
  pp.~111--139.

\bibitem{B62}
{\sc B.~T. Benjamin}, {\em {The solitary wave on a stream with an arbitrary
  distribution of vorticity}}, J. Fluid Mech., 12 (1962), pp.~97--116.

\bibitem{WC16}
{\sc R.~M. Chen and S.~Walsh}, {\em Continuous dependence on the density for
  stratified steady water waves}, Arch. Ration. Mech. Anal., 219 (2016),
  pp.~741--792.

\bibitem{WAl18b}
\leavevmode\vrule height 2pt depth -1.6pt width 23pt, {\em Unique determination
  of stratified steady water waves from pressure}, J. Differential Equations,
  264 (2018), pp.~115--133.

\bibitem{WAl18}
{\sc R.~M. Chen, S.~Walsh, and M.~H. Wheeler}, {\em Existence and qualitative
  theory for stratified solitary water waves}, Ann. Inst. H. Poincar\'{e} Anal.
  Non Lin\'{e}aire, 35 (2018), pp.~517--576.

\bibitem{C01a}
{\sc A.~Constantin}, {\em {Edge waves along a sloping beach}}, J. Phys. A, 34
  (2001), p.~9723.

\bibitem{Co01}
\leavevmode\vrule height 2pt depth -1.6pt width 23pt, {\em {On the deep water
  wave motion}}, J. Phys. A, 34 (2001), p.~1405.

\bibitem{Con11}
\leavevmode\vrule height 2pt depth -1.6pt width 23pt, {\em {Nonlinear Water
  Waves with Applications to Wave-Current Interactions and Tsunamis}}, vol.~81
  of {CBMS-NSF Conference Series in Applied Mathematics}, SIAM, Philadelphia,
  2011.

\bibitem{Con13}
\leavevmode\vrule height 2pt depth -1.6pt width 23pt, {\em Some
  three-dimensional nonlinear equatorial flows}, Journal of Physical
  Oceanography, 43 (2013), pp.~165--175.

\bibitem{EC11}
{\sc A.~Constantin and J.~Escher}, {\em {Analyticity of periodic traveling free
  surface water waves with vorticity}}, Ann. of Math., 173 (2011),
  pp.~559--568.

\bibitem{CI15}
{\sc A.~Constantin and R.~I. Ivanov}, {\em A hamiltonian approach to
  wave-current interactions in two-layer fluids}, Physics of Fluids, 27 (2015),
  p.~086603.

\bibitem{CIM16}
{\sc A.~Constantin, R.~I. Ivanov, and C.-I. Martin}, {\em Hamiltonian
  formulation for wave-current interactions in stratified rotational flows},
  Arch. Ration. Mech. Anal., 221 (2016), pp.~1417--1447.

\bibitem{CJ12}
{\sc A.~Constantin and R.~S. Johnson}, {\em An exact solution for equatorially
  trapped waves}, Journal of Geophysical Research: Oceans, 117 (2012),
  p.~C05029.

\bibitem{CJ16}
\leavevmode\vrule height 2pt depth -1.6pt width 23pt, {\em An exact, steady,
  purely azimuthal equatorial flow with a free surface}, Journal of Physical
  Oceanography, 46 (2016), pp.~1935--1945.

\bibitem{CS04}
{\sc A.~Constantin and W.~Strauss}, {\em {Exact steady periodic water waves
  with vorticity}}, Comm. Pure Appl. Math., 57 (2004), pp.~481--527.

\bibitem{CS11}
\leavevmode\vrule height 2pt depth -1.6pt width 23pt, {\em {Periodic traveling
  gravity water waves with discontinuous vorticity}}, Arch. Ration. Mech.
  Anal., 202 (2011), pp.~133--175.

\bibitem{CR71}
{\sc M.~G. Crandall and P.~H. Rabinowitz}, {\em {Bifurcation from simple
  eigenvalues}}, J. Functional Analysis, 8 (1971), pp.~321--340.

\bibitem{CB94}
{\sc B.~Cushman-Roisin and J.-M. Beckers}, {\em {Introduction to Geophysical
  Fluid Dynamics}}, Academic Press, 2009.

\bibitem{DJ32}
{\sc M.-L. Dubreil-Jacotin}, {\em {Sur les ondes de type permanent dans les
  liquides h\'et\'erog\`enes}}, Atti Accad. Naz. Lincei Rend., 6 (1932),
  pp.~814--819.

\bibitem{DJ37}
\leavevmode\vrule height 2pt depth -1.6pt width 23pt, {\em {Sur les
  th\'eor\`emes d'\`existence relatifs aux ondes permanentes p\'eriodiques a
  deux dimensions dans les liquides h\'et\'erog\`enes}}, J. Math. Pures. Appl.,
  9 (1937), pp.~43--67.

\bibitem{EW15}
{\sc M.~Ehrnstr\"{o}m and E.~Wahl\'{e}n}, {\em Trimodal steady water waves},
  Arch. Ration. Mech. Anal., 216 (2015), pp.~449--471.

\bibitem{EMM11}
{\sc J.~Escher, A.-V. Matioc, and B.-V. Matioc}, {\em {On stratified steady
  periodic water waves with linear density distribution and stagnation
  points}}, J. Differential Equations, 251 (2011), pp.~2932--2949.

\bibitem{EM14}
{\sc J.~Escher and B.-V. Matioc}, {\em On the analyticity of periodic gravity
  water waves with integrable vorticity function}, Diff. Int Equs., 27 (2014),
  pp.~217--232.

\bibitem{Ge09}
{\sc F.~Gerstner}, {\em {Theorie der {W}ellen samt einer daraus abgeleiteten
  {T}heorie der {D}eichprofile}}, Ann. Phys., 2 (1809), pp.~412--445.

\bibitem{GT01}
{\sc D.~Gilbarg and N.~S. Trudinger}, {\em {Elliptic Partial Differential
  Equations of Second Order}}, Springer Verlag, 2001.

\bibitem{HM18}
{\sc D.~Henry and C.~I. Martin}, {\em Exact, purely azimuthal stratified
  equatorial flows in cylindrical coordinates}, Dyn. Partial Differ. Equ., 15
  (2018), pp.~337--349.

\bibitem{HM19a}
\leavevmode\vrule height 2pt depth -1.6pt width 23pt, {\em Exact, free-surface
  equatorial flows with general stratification in spherical coordinates},
  Archive for Rational Mechanics and Analysis, 233 (2019), pp.~497--512.

\bibitem{HM19}
\leavevmode\vrule height 2pt depth -1.6pt width 23pt, {\em Free-surface, purely
  azimuthal equatorial flows in spherical coordinates with stratification}, J.
  Differential Equations, 266 (2019), pp.~6788--6808.

\bibitem{HM14b}
{\sc D.~Henry and A.-V. Matioc}, {\em {Global bifurcation of capillary-gravity
  stratified water waves}}, Proc. Roy. Soc. Edinburgh Sect. A, 144 (2014),
  pp.~775--786.

\bibitem{HM}
{\sc D.~Henry and B.-V. Matioc}, {\em {On the regularity of steady periodic
  stratified water waves}}, Commun. Pure Appl. Anal., 11 (2012),
  pp.~1453--1464.

\bibitem{HM13}
\leavevmode\vrule height 2pt depth -1.6pt width 23pt, {\em {On the existence of
  steady periodic capillary-gravity stratified water waves}}, Ann. Scuola Norm.
  Sup. Pisa, XII (2013), pp.~955--974.

\bibitem{MJ89}
{\sc M.~C.~W. Jones}, {\em Small amplitude capillary-gravity waves in a channel
  of finite depth}, Glasgow Math. J., 31 (1989), pp.~141--160.

\bibitem{Kl19}
{\sc M.~Kluczek}, {\em Exact {P}ollard-like internal water waves}, J. Nonlinear
  Math. Phys., 26 (2019), pp.~133--146.

\bibitem{L53}
{\sc R.~R. Long}, {\em {Some aspects of the flow of stratified fluids. {I}. {A}
  theoretical investigation}}, Tellus, 5 (1953), pp.~42--58.

\bibitem{CM13a}
{\sc C.~I. Martin}, {\em {Local bifurcation and regularity for steady periodic
  capillary-gravity water waves with constant vorticity}}, Nonlinear Anal. Real
  World Appl., 14 (2013), pp.~131--149.

\bibitem{Mar16}
\leavevmode\vrule height 2pt depth -1.6pt width 23pt, {\em Hamiltonian
  structure for rotational capillary waves in stratified flows}, J.
  Differential Equations, 261 (2016), pp.~373--395.

\bibitem{Mar17}
\leavevmode\vrule height 2pt depth -1.6pt width 23pt, {\em A {H}amiltonian
  approach for nonlinear rotational capillary-gravity water waves in stratified
  flows}, Discrete Contin. Dyn. Syst., 37 (2017), pp.~387--404.

\bibitem{CM14b}
{\sc C.~I. Martin and B.-V. Matioc}, {\em {Steady periodic water waves with
  unbounded vorticity: equivalent formulations and existence results}}, J.
  Nonlinear Sci., 24 (2014), pp.~633--659.

\bibitem{Ma12a}
{\sc A.-V. Matioc}, {\em {Steady internal water waves with a critical layer
  bounded by the wave surface.}}, J. Nonlinear Math. Phys., 19 (2012),
  pp.~1250008, 21 p.

\bibitem{MA13}
\leavevmode\vrule height 2pt depth -1.6pt width 23pt, {\em Exact geophysical
  waves in stratified fluids}, Appl. Anal., 92 (2013), pp.~2254--2261.

\bibitem{MM14}
{\sc A.-V. Matioc and B.-V. Matioc}, {\em {Capillary-gravity water waves with
  discontinuous vorticity: existence and regularity results}}, Comm. Math.
  Phys., 330 (2014), pp.~859--886.

\bibitem{M14}
{\sc B.-V. Matioc}, {\em {Global bifurcation for water waves with capillary
  effects and constant vorticity}}, Monatsh. Math., 174 (2014), pp.~459--475.

\bibitem{Pe79}
{\sc J.~Pedlosky}, {\em {Geophysical Fluid Dynamics}}, Springer, New York,
  1979.

\bibitem{P13}
{\sc P.~R. Pinet}, {\em {Invitation to Oceanography}}, Jones \& Bartlett
  Publishers, 2013.

\bibitem{Ra11}
{\sc R.~Stuhlmeier}, {\em {On edge waves in stratified water along a sloping
  beach}}, J. Nonlinear Math. Phys, 18 (2011), pp.~127--137.

\bibitem{T80}
{\sc R.~E.~L. Turner}, {\em {Internal waves in fluids with rapidly varying
  density}}, Ann. Scuola Norm. Sup. Pisa Cl. Sci. (4), 8 (1981), pp.~513--573.

\bibitem{W06b}
{\sc E.~Wahlén}, {\em {Steady periodic capillary-gravity waves with
  vorticity}}, SIAM J. Math. Anal., 38 (2006), pp.~921--943 (electronic).

\bibitem{Wal09a}
{\sc S.~Walsh}, {\em {Some criteria for the symmetry of stratified water
  waves}}, Wave Motion, 46 (2009), pp.~350--362.

\bibitem{Wal09}
\leavevmode\vrule height 2pt depth -1.6pt width 23pt, {\em {Stratified steady
  periodic water waves}}, SIAM J. Math. Anal., 41 (2009), pp.~1054--1105.

\bibitem{Wa14a}
\leavevmode\vrule height 2pt depth -1.6pt width 23pt, {\em Steady stratified
  periodic gravity waves with surface tension {I}: {L}ocal bifurcation},
  Discrete Contin. Dyn. Syst., 34 (2014), pp.~3241--3285.

\bibitem{Wal14b}
\leavevmode\vrule height 2pt depth -1.6pt width 23pt, {\em {Steady stratified
  periodic gravity waves with surface tension {II}: Global bifurcation}},
  Discrete Contin. Dyn. Syst., 34 (2014), pp.~3287--3315.

\bibitem{W13}
{\sc L.-J. Wang}, {\em {Regularity of traveling periodic stratified water waves
  with vorticity}}, Nonlinear Anal., 81 (2013), pp.~247--263.

\bibitem{Y60}
{\sc C.-S. Yih}, {\em {Exact solutions for steady two-dimensional flow of a
  stratified fluid.}}, J. Fluid Mech., 9 (1960), pp.~161--174.

\end{thebibliography}
\end{document}